\documentclass[11pt,a4paper,review]{siamart171218}
\usepackage[utf8]{inputenc}
\usepackage[english]{babel}

\usepackage{graphicx}
\usepackage{fancyhdr}

\usepackage{amsfonts}
\usepackage{amssymb}

\usepackage{amsmath}
\usepackage{amsfonts}
\usepackage{amssymb}
\usepackage{mathtools}
\usepackage{enumitem}
\usepackage[percent]{overpic}
\usepackage{tikz}
\usepackage{mathrsfs}
\usepackage{xargs}
\usepackage{xcolor}
\usepackage[colorinlistoftodos,prependcaption,textsize=tiny]{todonotes}
\newcommandx{\at}[2][1=]{\todo[linecolor=red,backgroundcolor=red!25,bordercolor=red,#1]{#2}}

\usepackage{algorithm}
\usepackage{algorithmic}

\usepackage{mathtools,booktabs}

\usepackage{varwidth}
\usepackage{hyperref}
\usepackage{cleveref}
\usepackage{cite}

\newcommand{\rr}[1]{\textcolor{red}{#1}}

\newcommand{\ignore}[1]{}

\title{Twice is enough for dangerous eigenvalues\thanks{Submitted to the editors \today.
\funding{The work of the first author was partially supported by NSF DMS-1818757.}}}
\author{Andrew Horning\thanks{Center for Applied Mathematics, Cornell University, Ithaca, NY 14853. (\email{ajh326@cornell.edu})} \and Yuji Nakatsukasa \thanks{Mathematical Institute, University of Oxford, Oxford, OX2 6GG. (\email{nakatsukasa@maths.ox.ac.uk})}}
\headers{Dangerous eigenvalues}{Andrew Horning and Yuji Nakatsukasa}
\begin{document}
\maketitle

\begin{abstract}
We analyze the stability of a class of eigensolvers that target interior eigenvalues with rational filters. We show that subspace iteration with a rational filter is robust even when an eigenvalue is near a filter's pole. These dangerous eigenvalues contribute to large round-off errors in the first iteration, but are self-correcting in later iterations. For matrices with orthogonal eigenvectors (e.g., real-symmetric or complex Hermitian), two iterations is enough to reduce round-off errors to the order of the unit-round off. In contrast, Krylov methods accelerated by rational filters with fixed poles typically fail to converge to unit round-off accuracy when an eigenvalue is close to a pole. In the context of Arnoldi with shift-and-invert enhancement, we demonstrate a simple restart strategy that recovers full precision in the target eigenpairs.
\end{abstract}



\section{Introduction}\label{sec:intro}

When combined with shift-and-invert enhancement, subspace iteration and Arnoldi are two classic iterative schemes for computing a few interior eigenvalues of an $n\times n$ matrix $A$. Each method constructs an orthonormal basis for a search subspace by iteratively applying the spectral filter
\begin{equation}\label{eqn:shift-and-invert_filter}
s(A)=(zI-A)^{-1}
\end{equation}
to a set of vectors. Approximate eigenpairs can then be extracted from the search subspace with a projection step, e.g., Rayleigh--Ritz. The shift $z$ is selected to target a region of interest, and both methods typically approximate eigenvalues of $A$ closest to $z$.

Recently, general rational filters of the form
\begin{equation}\label{eqn:rational_filter}
r(A)=\sum_{j=1}^\ell \omega_j(z_jI-A)^{-1},
\end{equation}
have attracted a great deal of attention in the context of large, data-sparse eigenvalue problems~\cite{sakurai2003projection,kestyn2016pfeast,polizzi2009density,peter2014feast,guttel2015zolotarev,austin2015computing,horning2020feast}. When the weights $\omega_1,\ldots,\omega_\ell$ and nodes $z_1,\ldots,z_\ell$ are chosen appropriately, these rational filters can robustly target eigenvalues in a region of interest and significantly accelerate the convergence of the subspaces constructed by subspace iteration, Arnoldi, or variants thereof~\cite{peter2014feast,austin2015computing}. They also tend to be highly parallelizable because each shift-and-invert transformation may be applied independently~\cite{kestyn2016pfeast}.

In his 2001 volume on matrix algorithms for eigenvalue problems, Stewart noted that shift-and-invert Arnoldi encounters difficulties in floating-point arithmetic when the shift lies too close to an eigenvalue of $A$\cite[p.~309]{stewart2001matrix}. Although the eigenvalue adjacent to the shift is rapidly approximated to the order of the unit round-off $u$, the residuals of other computed eigenpairs stagnate near the order of $u/d$, where $d$ is the distance between the ``dangerous" eigenvalue and the shift. This phenomenon has also recently been observed in the context of Krylov methods, where the subspace is constructed with contour integrals and rational approximation~\cite{austin2015computing}.

Curiously, dangerous eigenvalues do not inflict the same stagnation in the residuals of the other target eigenpairs during subspace iteration.~\Cref{fig:arnoldi_fsi} compares the residuals of two target eigenpairs computed with Arnoldi (left) and subspace iteration (right), using the shift-and-invert filter in~\cref{eqn:shift-and-invert_filter} with $z=10$. The approximation to the dangerous eigenvalue $\lambda_1=10+10^{-12}$ converges rapidly to unit round-off accuracy in both cases. However, only subspace iteration computes an approximation to the second target eigenvalue $\lambda_2\approx 10.1$ to unit round-off accuracy.

A similar story unfolds in~\cref{fig:contour_refine}, where we compute two target eigenpairs with the contour integral eigensolver described in~\cite{peter2014feast}, one of them located at a distance of $10^{-10}$ from the contour. As we refine the quadrature along the contour, the poles of a rational filter with form~\cref{eqn:rational_filter} cluster near the dangerous eigenvalue, and we observe the residual of the dangerous eigenpair converge rapidly to unit round-off, while the residuals of the remaining target pairs stagnate near $10^{-5}$. On the other hand, if we fix the number of quadrature points (i.e., poles) and refine via filtered subspace iteration, the residuals of all target eigenpairs converge geometrically to order $u$.

This paper is about explaining~\cref{fig:arnoldi_fsi,fig:contour_refine}. We first examine how rational filtered subspace iteration disarms dangerous eigenvalues after the first iteration. When $A$ has a complete set of orthonormal eigenvectors, orthogonal bases for the search subspace play a special role and ``twice-is-enough" to recover full precision in the computed iterates (see~\cref{sec:dang_eigvals,sec:twice_is_enough,sec:PSI_analysis}).\footnote{Aspects of our analysis are similar to Parlett and Kahan's ``twice-is-enough" algorithm and analysis for Gram-Schmidt reorthogonalization~\cite[pp.~107--109]{parlett1998symmetric}.} In the non-normal case, iterating on approximate eigenvectors (obtained from a Rayleigh--Ritz step, for instance) is the key to overcoming round-off errors incurred by the dangerous eigenvalue, while iterations based on orthogonal bases (such as approximate Schur vectors) suffer stagnation in the remaining target eigenpairs (see~\cref{sec:non-normal_case}).

\begin{figure}[!tbp]
  \centering
  \begin{minipage}[b]{0.48\textwidth}
    \begin{overpic}[width=\textwidth]{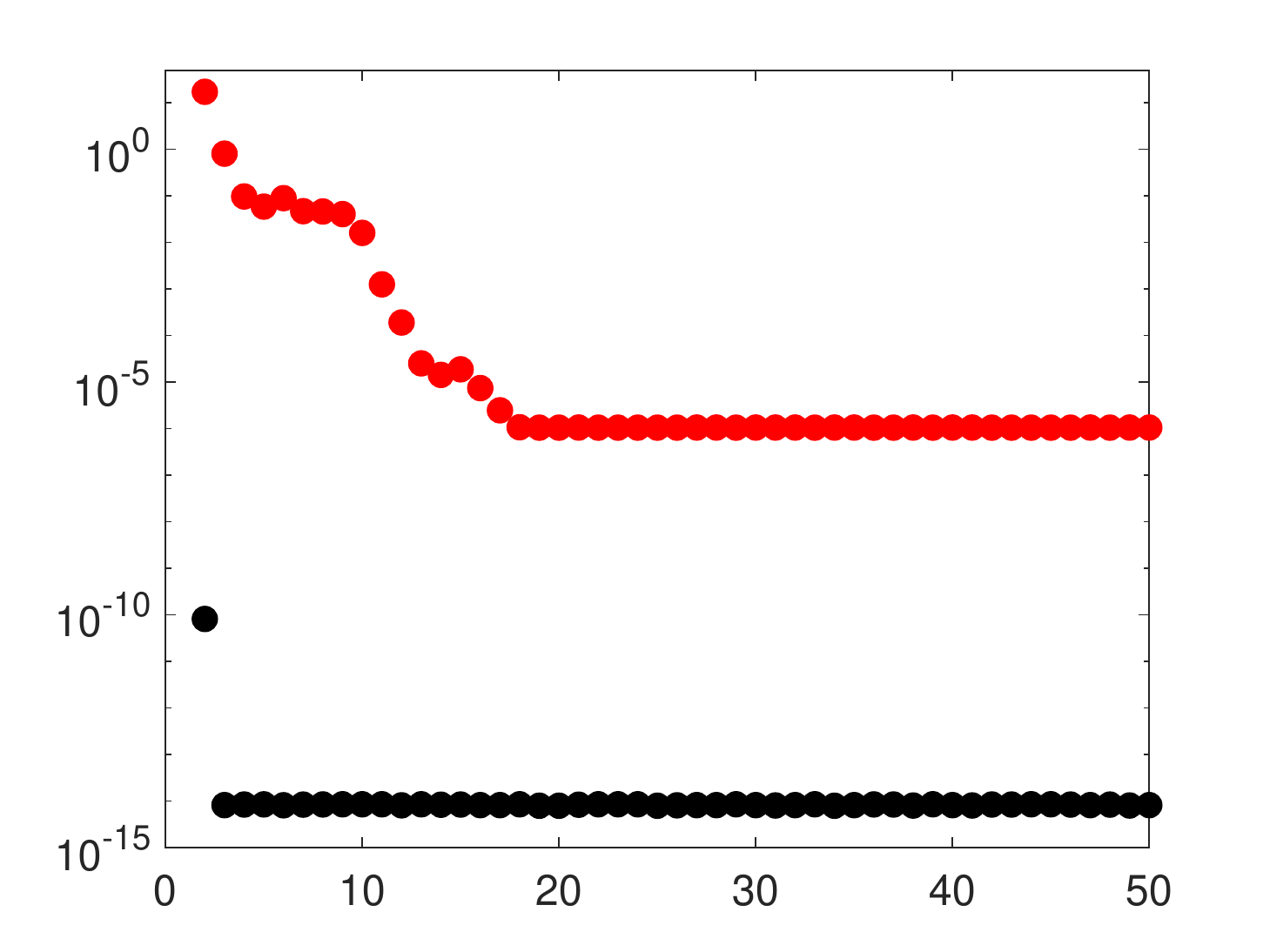}
     \put (34,73) {$\displaystyle \|A\hat v_i-\hat \lambda_i\hat v_i\|$}
     \put (75,14) {$\displaystyle i=1$}
   	 \put (75,44) {$\displaystyle i=2$}
     \put (50,-2) {$\displaystyle k$}
     \end{overpic}
  \end{minipage}
  \hfill
  \begin{minipage}[b]{0.48\textwidth}
    \begin{overpic}[width=\textwidth]{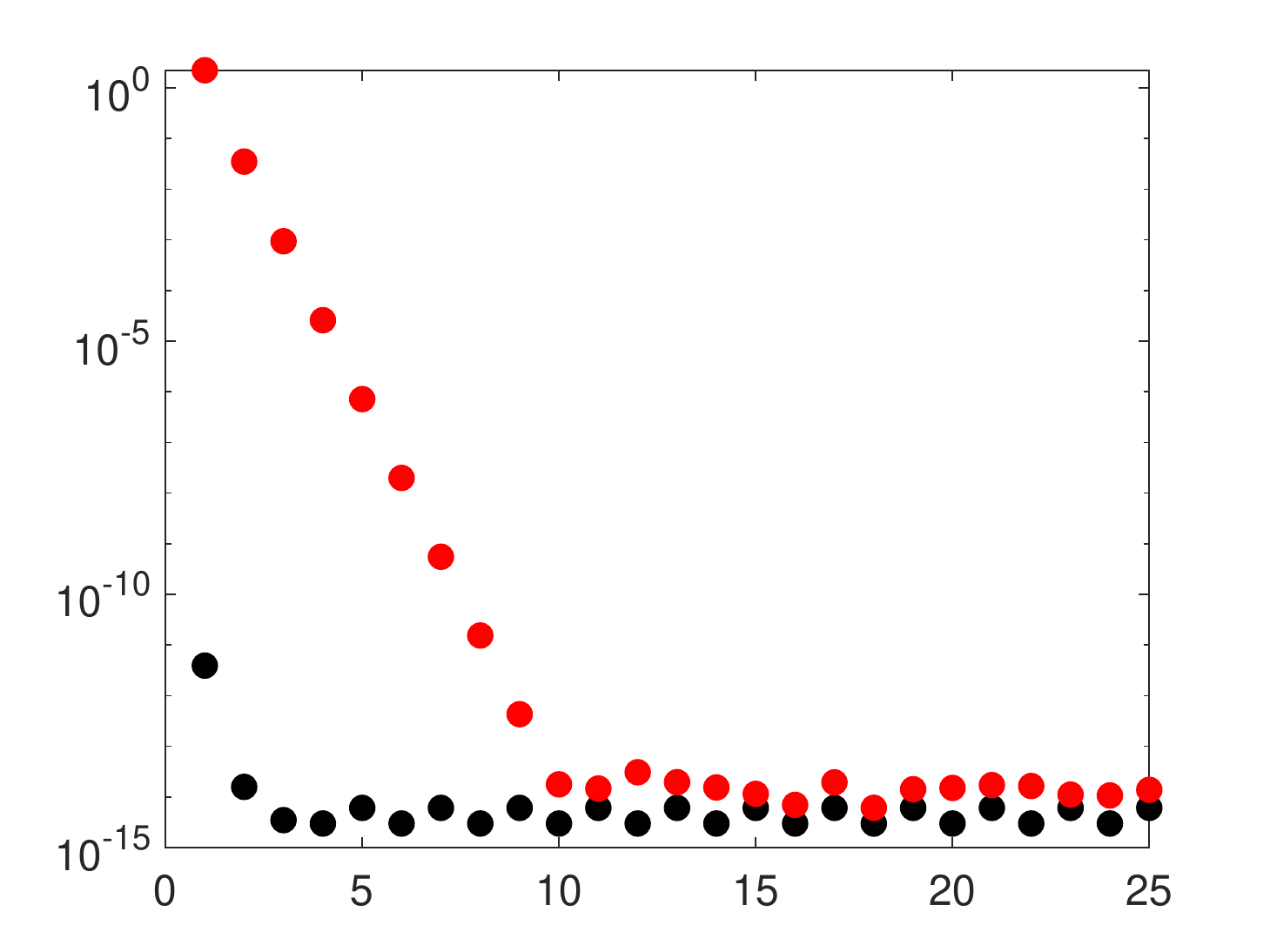}
     \put (34,73) {$\displaystyle \|A\hat v_i-\hat \lambda_i\hat v_i\|$}
     \put (22,14) {$\displaystyle i=1$}
   	 \put (25,55) {\rotatebox{-58}{$\displaystyle i=2$}}
     \put (50,-2) {$\displaystyle k$}
     \end{overpic}
  \end{minipage}
  \caption{
The residuals for two approximate eigenpairs of a real-symmetric $100\times 100$ matrix at iterations $k=2,\ldots,50$ of Arnoldi (left) and iterations $k=1,\ldots,25$ of subspace iteration (right), both with shift-and-invert enhancement. The approximate eigenpairs correspond to a dangerous eigenvalue (black) with $|z-\lambda_1|=10^{-12}$ and a second target eigenvalue (red) with $|z-\lambda_2|\approx 0.1$.
\label{fig:arnoldi_fsi}}
\end{figure}
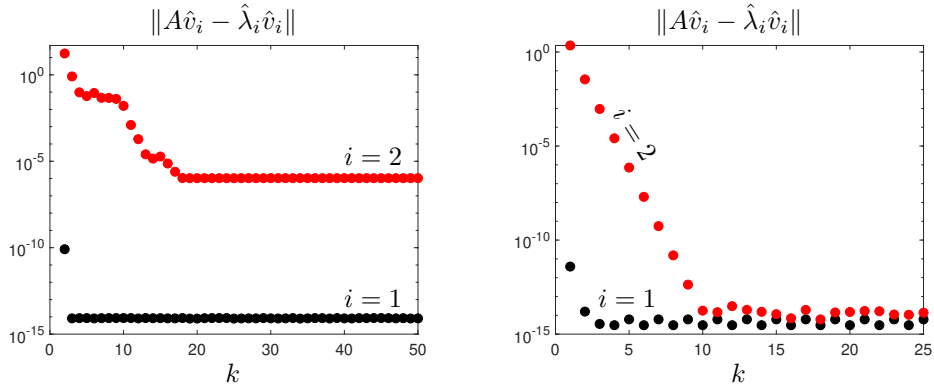

To obtain full precision in the remaining target eigenpairs for Arnoldi and related Krylov schemes, the prevailing consensus is to alter the rational filter by moving or removing the offending poles~\cite{stewart2001matrix,austin2015computing}. Unfortunately, this usually means settling for a less efficient filter or starting over with a new filter. Informed by our analysis of subspace iteration and its immunity to dangerous eigenvalues, we offer simple restart strategies that fix stagnation in shift-and-invert Arnoldi (see~\cref{sec:rat_Krylov}).

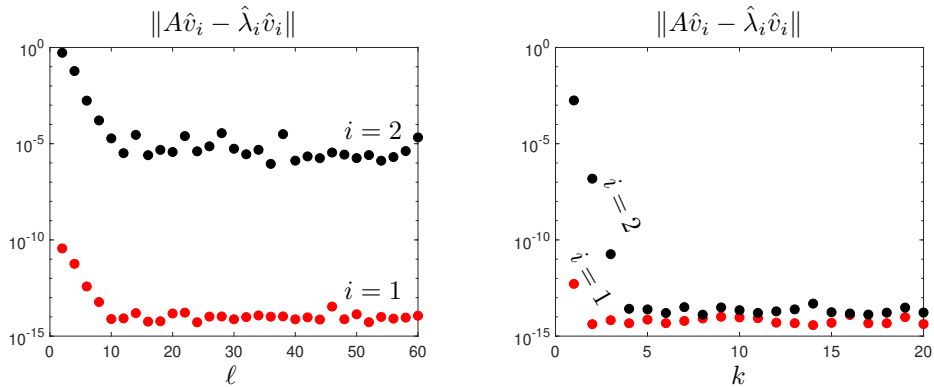
\begin{figure}[!tbp]
  \centering
  \begin{minipage}[b]{0.48\textwidth}
    \begin{overpic}[width=\textwidth]{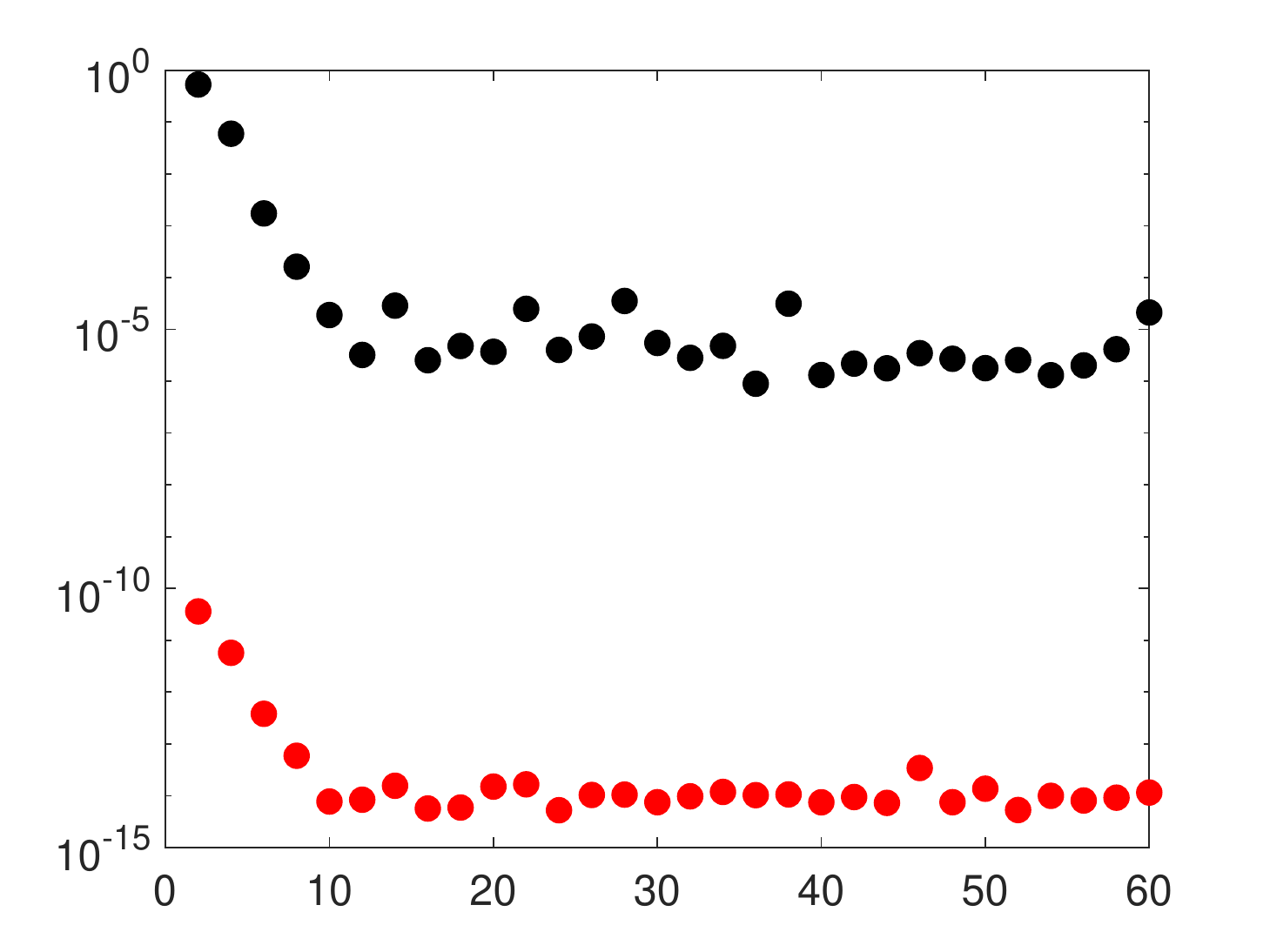}
     \put (34,73) {$\displaystyle \|A\hat v_i-\hat \lambda_i\hat v_i\|$}
     \put (75,16) {$\displaystyle i=1$}
   	 \put (75,50) {$\displaystyle i=2$}
     \put (50,-2) {$\displaystyle \ell$}
     \end{overpic}
  \end{minipage}
  \hfill
  \begin{minipage}[b]{0.48\textwidth}
    \begin{overpic}[width=\textwidth]{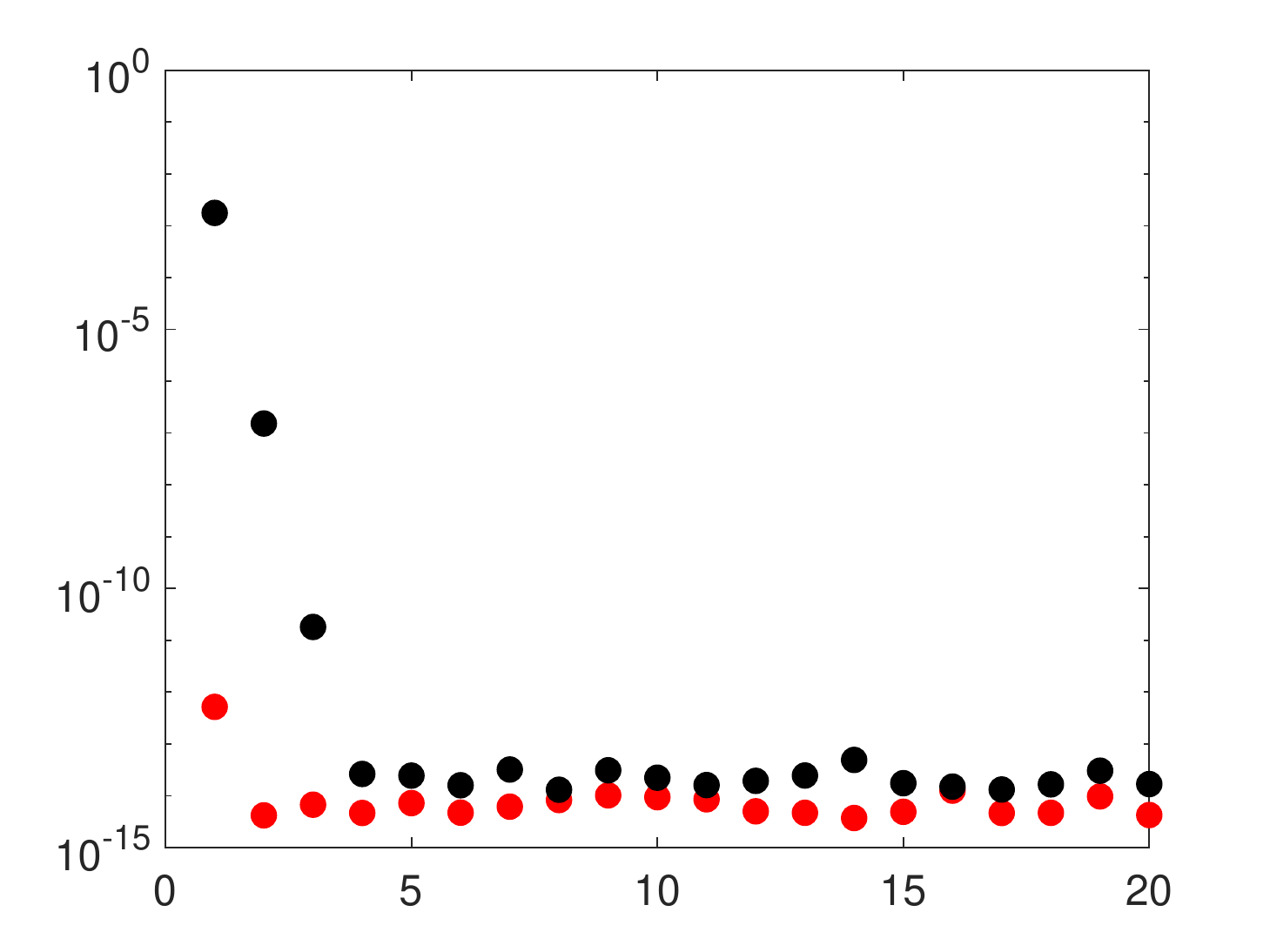}
     \put (34,73) {$\displaystyle \|A\hat v_i-\hat \lambda_i\hat v_i\|$}
     \put (16,25) {\rotatebox{-60}{$\displaystyle i=1$}}
   	 \put (23,41) {\rotatebox{-66}{$\displaystyle i=2$}}
     \put (50,-2) {$\displaystyle k$}
     \end{overpic}
  \end{minipage}
  \caption{The left panel displays the residuals for two approximate eigenpairs of a $100\times 100$ real-symmetric matrix computed with the contour integral eigensolver described in~\cite{peter2014feast}, as the quadrature rule approximating the contour integral is refined. One of the target eigenvalues ($i=1$, black) is a distance of $10^{-10}$ from the contour. The right panel displays the residuals for the two approximate eigenpairs when refined via iteration rather than quadrature rule. The quadrature rule used corresponds to a rational filter in~\cref{eqn:rational_filter} with $\ell=8$.
\label{fig:contour_refine}}
\end{figure}

Our analysis is focused on a matrix $A$ with a single dangerous eigenvalue located at a distance $d\ll 1$ from a pole of the filter in~\cref{eqn:rational_filter}. To reveal the precise influence of the dangerous eigenvalue, we frame our discussion in the asymptotic limit $d\rightarrow 0$. However, we always provide concrete bounds and give leading order estimates to elucidate the role of salient parameters, e.g., related to the rational filter or non-normality of $A$. We consider the implications of our results for other natural configurations, such as multiple eigenvalues clustered at a pole, in~\cref{sec:num_exp}.

Throughout the paper, $\|\cdot\|$ denotes the spectral norm of a matrix (Euclidean norm for vectors) and $A$ denotes an $n\times n$ diagonalizable matrix with eigenvalues and eigenvectors satisfying $Av_i=\lambda_iv_i$, for $1\leq i\leq n$. Except in~\cref{sec:non-normal_case}, we assume that $A$ has a complete orthonormal set of eigenvectors (i.e., $A$ is normal), in which case it is convenient to write the eigendecomposition of $A$ in the form
\begin{equation}\label{eqn:eigendecomposition}
A=V_1\Lambda_1V_1^* + V_2\Lambda_2V_2^*.
\end{equation}
Here, $\Lambda_1={\rm diag}(\lambda_1,\ldots,\lambda_m)$ contains a set of target eigenvalues that we wish to compute and $\Lambda_2={\rm diag}(\lambda_{m+1},\ldots,\lambda_n)$ contains the remaining unwanted eigenvalues (usually, $m\ll n$). We denote the target eigenspace by $\mathcal{V}={\rm span}(V_1)$ and the full spectrum of $A$ by $\Lambda=\Lambda_1\bigcup\Lambda_2$. 

For simplicity, we always assume that $r(\Lambda)$ is invertible, that there is a nonzero spectral gap between $r(\Lambda_1)$ and $r(\Lambda_2)$, and index the eigenvalues in order of decreasing modulus under the filter so that
\begin{equation}\label{eqn:eig_index}
|r(\lambda_1)|\geq\cdots\geq|r(\lambda_m)| > |r(\lambda_{m+1})| \geq \cdots \geq |r(\lambda_n)|.
\end{equation}
Here, $r(\lambda)=\sum_{j=1}^\ell \omega_j(z_j-\lambda)^{-1}$ is the scalar form of the filter in~\cref{eqn:rational_filter}.\footnote{We refer to the scalar function $r(z)$ and its matrix companion $r(A)$ with the same symbol. We always include the argument when it is necessary to clarify which we mean.} Under the ordering in~\eqref{eqn:eig_index}, the dangerous eigenvalue is $\lambda_1$. Without loss of generality, we assume that the weight $w_j$ associated with a pole near the dangerous eigenvalue $\lambda_1$ is equal to one (by scaling $r(\cdot)$ if necessary). This simplifies the analysis and usually implies that the other weights $w_i$ are also modest in size. Finally, we tacitly assume $\|A\|=\mathcal{O}(1)$ in our informal discussions; the formal theorems and statements hold without this assumption.

\section{Subspace iteration with rational filters}\label{sec:fsi}

Given an $n\times m$ matrix $Q_0$ with orthonormal columns, the simplest practical form of subspace iteration with a rational filter, as in~\cref{eqn:rational_filter}, computes the iterates
\begin{equation}\label{eqn:ratSI}
X_k=r(A)Q_{k-1},\qquad Q_k={\rm qf}(X_k).
\end{equation}
Here, ${\rm qf}(X_k)$ denotes the orthogonal factor from a QR decomposition of $X_k$. The eigenvalues of $Q_k^*AQ_k$ provide approximations to the target eigenvalues, and approximate eigenvectors are given by $Q_kx_i$ for each eigenvector, $x_i$, of the small $m\times m$ matrix $Q_k^*AQ_k$. These approximations to the target eigenpairs are called Ritz pairs.

Intuitively, the Ritz pairs extracted with the basis $Q_k$ are usually good approximations to the target eigenpairs when there are good approximations to $v_1,\ldots,v_m$ in $\mathcal{S}_k={\rm span}(Q_k)$. Here, the rational filter in~\cref{eqn:ratSI} fills two complementary roles. First, the filter should guide the iterates toward the target eigenspace by mapping the target eigenvalues of $A$ to the dominant eigenvalues of $r(A)$ (that is, the eigenvalues with the largest modulus $|r(\lambda_i)|$). Second, the filter should enhance the gap between the target eigenvalues and the unwanted eigenvalues to accelerate the convergence of the Ritz pairs. These criteria follow from a standard one-step refinement bound for subspace iteration~\cite[Thm. 5.2]{saad2011numerical}.
\begin{theorem}\label{thm:1step_standard}
Let normal $A\in\mathbb{C}^{n\times n}$ and $r:\Lambda\rightarrow\mathbb{C}$ satisfy~\cref{eqn:eigendecomposition,eqn:eig_index}, respectively, and let $\mathcal{S}_{j}={\rm span}(Q_{j})$ in~\cref{eqn:ratSI}, for $j\geq 0$. If $V_1^*Q_0$ has full rank, then for each $v_i\in\mathcal{V}$ 
there are vectors $s_i^{(j)}\in\mathcal{S}_{j}$ such that 
\begin{equation}\label{eqn:1step_standard}
\|s_i^{(k)}-v_i\|\leq\Big\lvert\frac{r(\lambda_{m+1})}{r(\lambda_i)}\Big\rvert \|s_i^{(k-1)}-v_i\|\leq \Big\lvert\frac{r(\lambda_{m+1})}{r(\lambda_i)}\Big\rvert^k \|s_i^{(0)}-v_i\|.
\end{equation}
Moreover, each $P_\mathcal{V}s_i^{(j)}=v_i$, where $P_\mathcal{V}=V_1V_1^*$ is the spectral projector onto $\mathcal{V}$.
\end{theorem}

\Cref{thm:1step_standard} implies that there are approximations in $\mathcal{S}_k$ that converge geometrically to the $i$th target eigenvector with rate $|r(\lambda_{m+1})|/|r(\lambda_i)|$.\footnote{If $A$ does not possess orthogonal eigenvectors, this rate is only asymptotic as $k\rightarrow\infty$ due to the phenomena of transient growth in matrix powers of non-normal matrices~\cite[Ch. 16]{trefethen2005spectra}.} Consequently, if the filter is very small on the unwanted eigenvalues relative to its magnitude on the target eigenvalues, then we expect the Ritz pairs to converge rapidly. The appeal of rational filters in the modern computing era is that filters of modest degree $\ell\leq 20$ often achieve $|r(\lambda_{m+1})|/|r(\lambda_m)|\approx u$. In a typical parallel computing environment, the individual shifted inverses in~\cref{eqn:rational_filter} are easily applied in parallel, meaning that the target eigenpairs can be computed to machine precision at the equivalent (serial) cost of solving a shifted linear system. However, a higher degree rational filter and multiple iterations may be required when many eigenvalues are clustered near the target group. Additionally, clustered eigenvalues may lead to ill-conditioned eigenvectors and loss of orthogonality in the Ritz pairs. When eigenvalues are clustered and more poles are employed in the rational filter, one may also encounter dangerous eigenvalues.

In practice, there are many modifications one can make to~\cref{eqn:ratSI} to improve convergence, enhance stability, or increase computational efficiency. Nevertheless, when $A$ is normal,~\cref{eqn:ratSI} is enough to capture both the dangers and the self-correcting effects of eigenvalues that are close to the poles in~\cref{eqn:rational_filter}. When $A$ is non-normal, iterations that incorporate the Ritz vectors when forming $Q_{k-1}$ play a special role, while other variants (including~\cref{eqn:ratSI} itself) typically fail to converge to full precision (see~\cref{fig:exp3_results}). We discuss these modifications further in~\cref{sec:non-normal_case}.

\subsection{Principal angles between subspaces}

The principal angles between the subspaces $\mathcal{S}_k$ and $\mathcal{V}$ provide a natural framework with which to characterize the refinement of the iterates in~\cref{eqn:ratSI}. Generalizing the notion of an angle between two vectors, the principal angles tell us how close $\mathcal{S}_k$ and $\mathcal{V}$ are in a geometric sense~\cite{bjorck1973numerical}. 
\begin{definition}\label{def:PABS}
Let $\mathcal{X}$ and $\mathcal{Y}$ be two $m$-dimensional subspaces with orthonormal bases $X$ and $Y$, respectively, and let $\sigma_i(Y^*X)$ denote the $i$th singular value of $Y^*X$. The principal angles between $\mathcal{X}$ and $\mathcal{Y}$ are the acute angles $\theta_1(\mathcal{X},\mathcal{Y})\geq\cdots\geq\theta_m(\mathcal{X},\mathcal{Y})$ satisfying
\begin{equation}\label{eqn:define_PABS}
\cos\theta_i(\mathcal{X},\mathcal{Y})=\sigma_{m+1-i}(Y^*X), \qquad i=1,\ldots,m.
\end{equation}
\end{definition}
The sine of the largest principal angle, given by $\sin\theta_1(\mathcal{X},\mathcal{Y})=\|(I-P_\mathcal{Y})X\|$, defines a metric on the set of $m$-dimensional subspaces. However, the tangents of the principal angles, which are the singular values of the matrix~\cite{zhu2013angles}
\begin{equation}\label{eqn:tan_PABS}
T(X,Y)=(I-P_\mathcal{Y})X(Y^*X)^+,
\end{equation} 
are better equipped to describe the behavior of the iterates in~\cref{eqn:ratSI}. In~\cref{eqn:tan_PABS}, $(Y^*X)^+$ denotes the Moore--Penrose pseudoinverse of $Y^*X$ and, crucially, $X$ need not be orthonormal.

A subspace analogue of~\cref{thm:1step_standard}, based on the largest principal angle between $\mathcal{S}_k$ and $\mathcal{V}$, is easy to derive with~\cref{eqn:tan_PABS}. 
\begin{theorem}\label{thm:standardSI_bound}
Let normal $A\in\mathbb{C}^{n\times n}$ and $r:\Lambda\rightarrow\mathbb{C}$ satisfy~\cref{eqn:eigendecomposition,eqn:eig_index}, respectively, 
and let $\mathcal{S}_{j}={\rm span}(Q_{j})$ in~\cref{eqn:ratSI}.
If $\cos\theta_1(\mathcal{S}_{0},\mathcal{V})>0$, then
\begin{equation}
\tan\theta_1(\mathcal{S}_k,\mathcal{V})\leq\Big\lvert\frac{r(\lambda_{m+1})}{r(\lambda_m)}\Big\rvert\tan\theta_1(\mathcal{S}_{k-1},\mathcal{V})
\leq \Big\lvert\frac{r(\lambda_{m+1})}{r(\lambda_m)}\Big\rvert^k\tan\theta_1(\mathcal{S}_{0},\mathcal{V}).
\end{equation}
\end{theorem}
\begin{proof}
We prove the first inequality with a direct calculation using~\cref{eqn:tan_PABS}; the second follows immediately by induction and the fact that $\cos\theta>0$ when $\tan\theta<\infty$. We compute that 
$(I-P_\mathcal{V})X_k = V_2r(\Lambda_2)V_2^*Q_{k-1}$ and that $V_1^*X_k=r(\Lambda_1)V_1^*Q_{k-1}$. Using the induction hypothesis that $\cos\theta_1(\mathcal{S}_{k-1},\mathcal{V})>0$, which implies $V_1^*Q_{k-1}$ is invertible, we obtain
\begin{equation}\label{eqn:compute_tangent}
(I-P_\mathcal{V})X_k(V_1^*X_k)^+=V_2r(\Lambda_2)V_2^*Q_{k-1}(V_1^*Q_{k-1})^{-1}r(\Lambda_1)^{-1}.
\end{equation}
The theorem follows by taking norms and noting that $\|V_2^*Q_{k-1}(V_1^*Q_{k-1})^{-1}\|=\tan\theta_1(\mathcal{S}_{k-1},\mathcal{V})$, $\|r(\Lambda_2)\|=|r(\lambda_{m+1})|$, and $\|r(\Lambda_1)^{-1}\|=|r(\lambda_m)|^{-1}$.
\end{proof}
We note that~\cref{thm:1step_standard} is recovered from~\cref{eqn:compute_tangent} by post-multiplying each side by the unit vector $e_i$ and setting $s_j=X_k(V_1^*X_k)^+e_i$, for $j=k-1,k$.

The tangents (and sines) of the principal angles play an important role in the perturbation theory of eigenpairs and, consequently, the bounds in~\cref{thm:standardSI_bound} are useful when determining the accuracy in the computed Ritz pairs~\cite{stewart1990matrix,stewart2001matrix,saad2011numerical}. For our purposes,~\cref{thm:standardSI_bound} and its proof are useful tools when analyzing subspace iterations subject to perturbations (see~\cref{sec:PSI_analysis}), because $\tan\theta_1(\mathcal{S}_k,\mathcal{V})$ is computed directly from the iterate $X_k$.

\section{Dangerous eigenvalues}\label{sec:dang_eigvals}

When an eigenvalue of $A$ is very close to a pole of the rational filter in~\cref{eqn:rational_filter}, $r(A)$ disproportionately amplifies components in the direction of the associated eigenvector. Given any vector $x\in\mathbb{C}^n$, we estimate
\begin{equation}\label{eqn:amplify_v1}
r(A)x = \sum_{i=1}^n r(\lambda_i)v_iv_i^*x = \frac{v_1^*x}{de^{i\theta}}v_1
 + \mathcal{O}(1), \quad\text{as}\quad d\rightarrow 0.
\end{equation}
(It is convenient to write the complex-valued difference between $\lambda_1$ and the nearest pole $z_{j_*}$ in the polar notation $z_{j_*}-\lambda_1=de^{i\theta}$, with argument $0\leq\theta<2\pi$.)
This amplification is precisely the reason that shift-and-invert power iterations are so effective when the shift is close to the target eigenvalue. If we apply $r(A)$ to a random vector with unit norm and normalize, the result approximates $v_1$ with relative accuracy $\mathcal{O}(d)$, under the generic assumption that the random vector is not nearly orthogonal to $v_1$. Similarly, when $r(A)$ is applied to a random orthonormal matrix $Q_0$, ${\rm span}(r(A)Q_0)$ contains good approximations to $v_1$ when $\|v_1^*Q_0\|$ is not too small.

\begin{figure}[!tbp]
  \centering
  \begin{minipage}[b]{0.98\textwidth}
    \begin{overpic}[width=\textwidth]{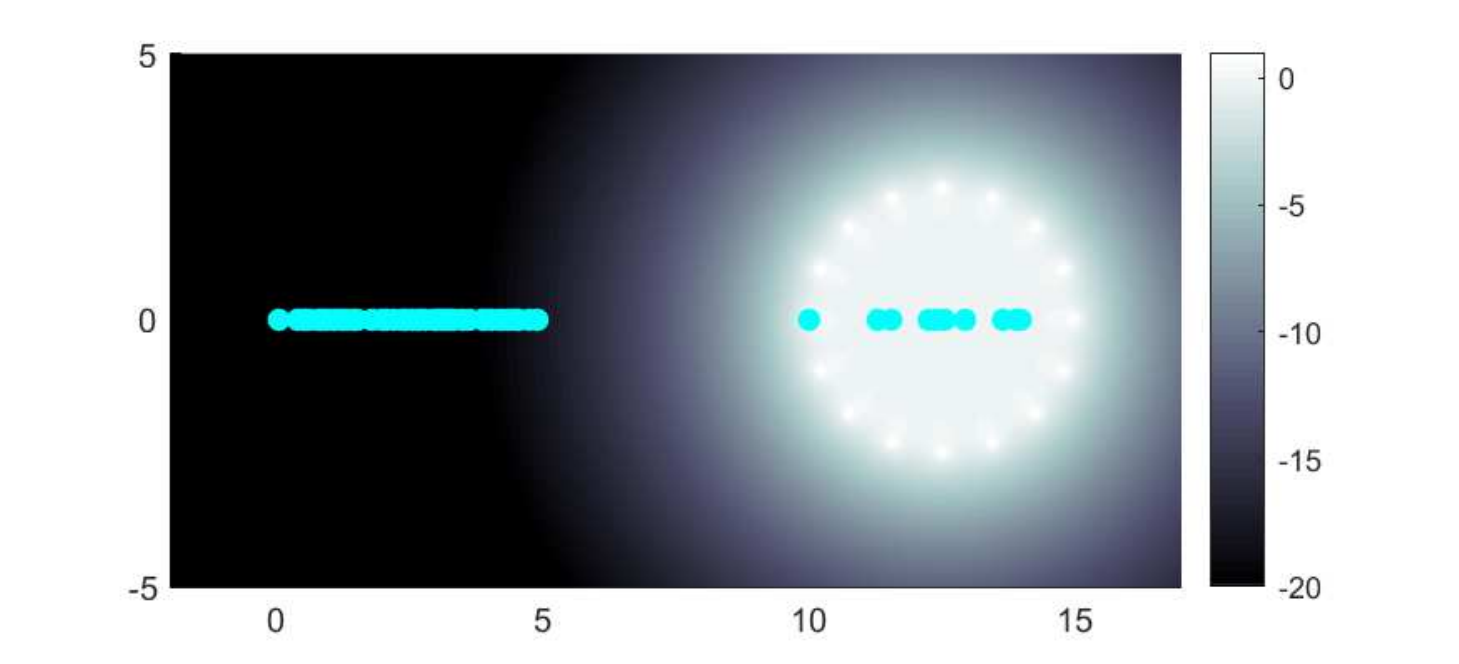}
   	 \put (5,20) {\rotatebox{90}{$\displaystyle {\rm Imag}(z)$}}
   	 \put (92,28) {\rotatebox{-90}{$\displaystyle \log|r(z)|$}}
     \put (44,0) {$\displaystyle {\rm Real}(z)$}
     \end{overpic}
  \end{minipage}
  \caption{The eigenvalues of a $100\times 100$ real-symmetric matrix overlaid on a complex color plot of the magnitude of a rational approximation to the characteristic function on $[10,15]$. A dangerous eigenvalue is located at distance $d=10^{-10}$ from the pole at $z=10$.
\label{fig:exp1_setup}}
\end{figure}

However, the amplifying effect of a dangerous eigenvalue may cause issues when computing the iterates in~\cref{eqn:ratSI} in floating-point arithmetic. \Cref{fig:exp1_setup} shows the eigenvalues of a $100\times 100$ real symmetric matrix plotted in the complex plane over the magnitude (indicated by color) of a rational filter targeting the interval $[10,15]$. The matrix has a large cluster of eigenvalues in the interval $[0,5]$, where the filter has decayed to less than unit round-off, and a small set of eigenvalues in the target region, where the filter has magnitude close to $1$. One eigenvalue of the matrix is very close to the pole at $z=10$, separated by a distance of $10^{-10}$. By~\cref{thm:1step_standard}, we expect that (in exact arithmetic) all of the eigenvalues in the target region are resolved to accuracy on the order of $u$ after one iteration. However,~\cref{fig:exp1_results} (left) shows that only the dangerous eigenpair has been computed accurately. The residuals of the remaining target eigenpairs are on the order of $10^{-5}$, that is, roughly $u/d$.

\begin{figure}[!tbp]
  \centering
  \begin{minipage}[b]{0.48\textwidth}
    \begin{overpic}[width=\textwidth]{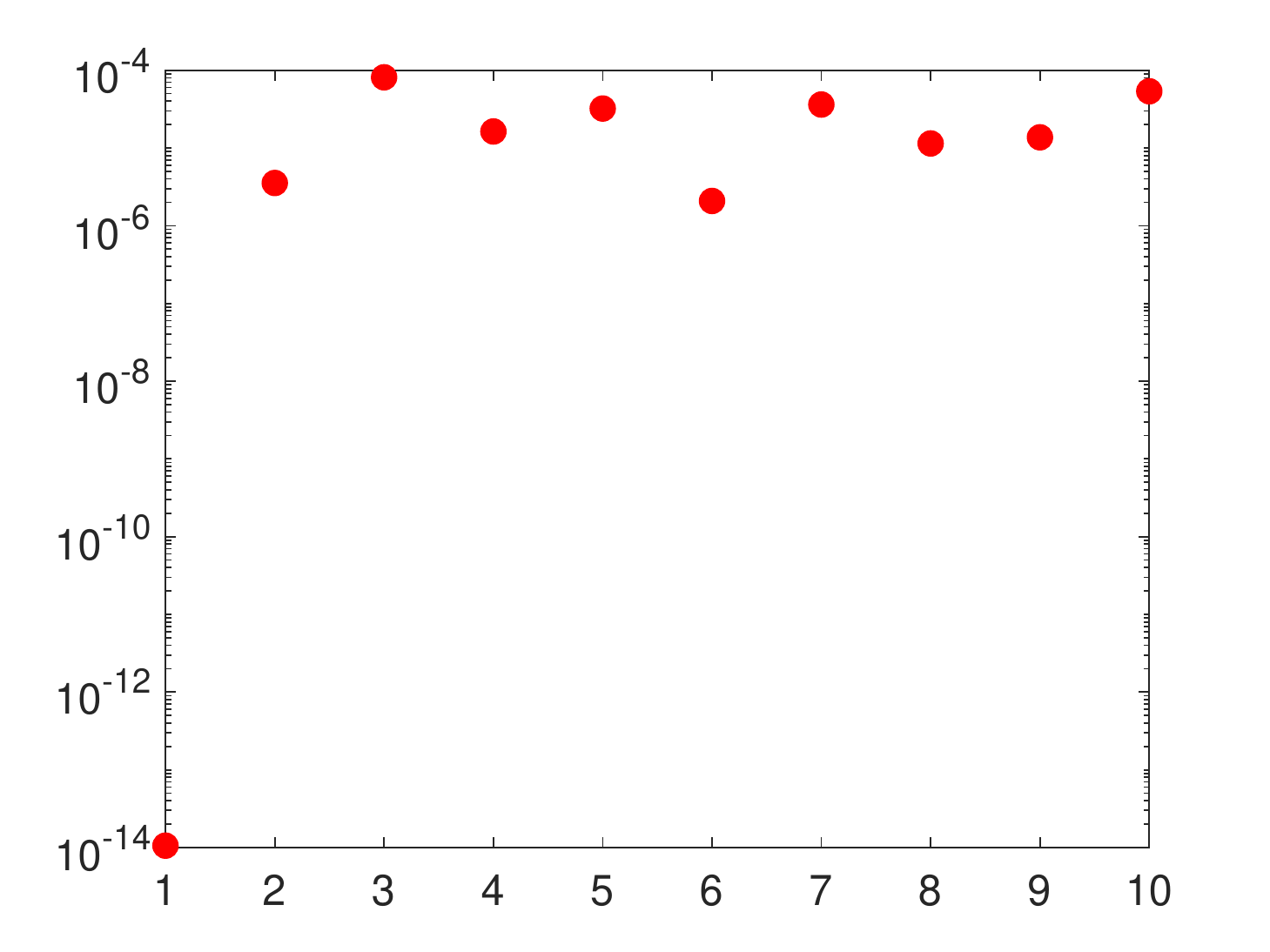}
   	 \put (34,73) {$\displaystyle \|A\hat v_i-\hat \lambda_i\hat v_i\|$}
     \put (50,-2) {$\displaystyle i$}
     \end{overpic}
  \end{minipage}
  \hfill
  \begin{minipage}[b]{0.48\textwidth}
    \begin{overpic}[width=\textwidth]{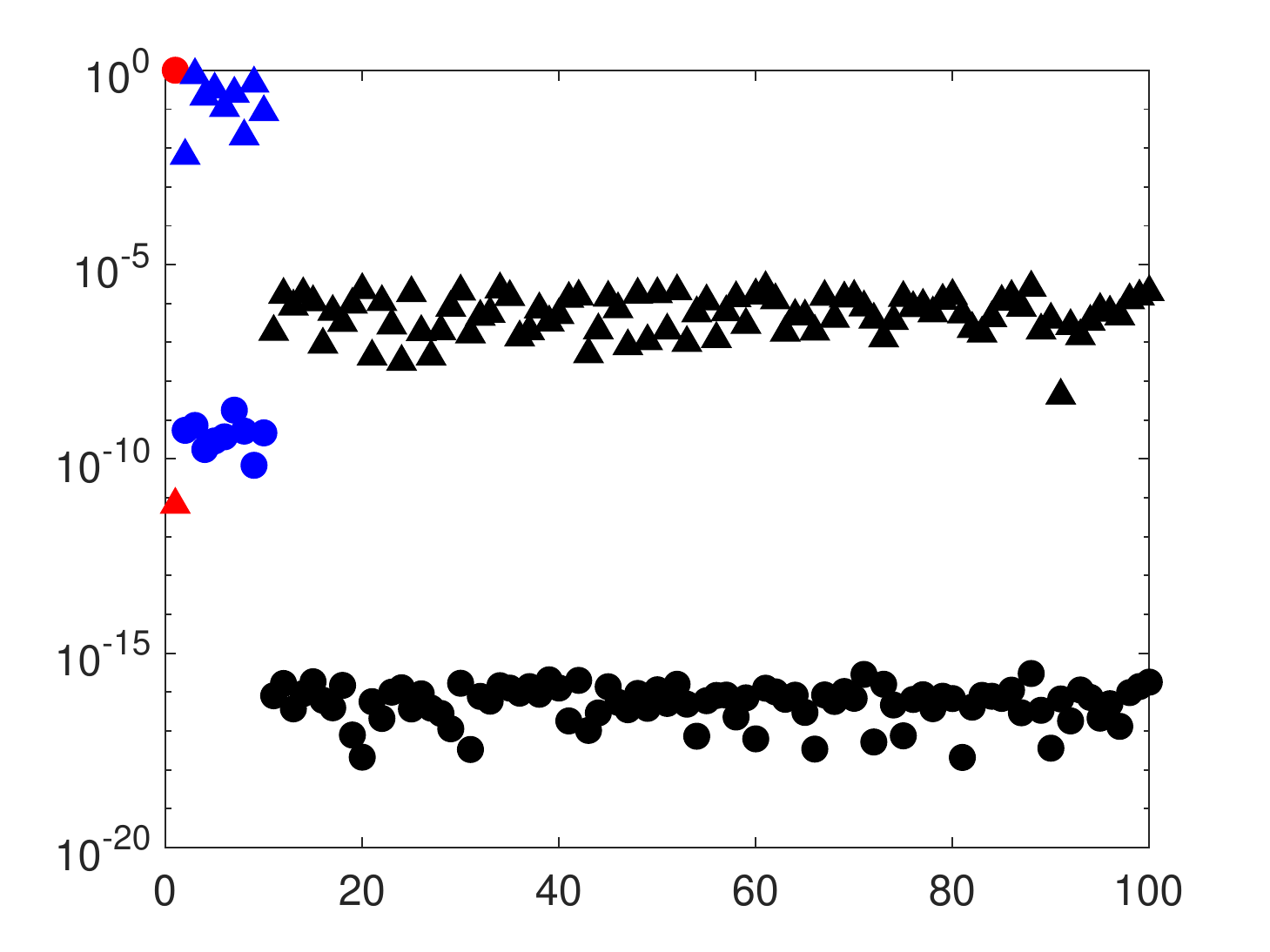}
     \put (43,73) {$\displaystyle |v_i^*\hat q^{(1)}_j|$}
     \put (50,-2) {$\displaystyle i$}
          \put (70,26) {$\displaystyle j=1$}
   	 \put (70,58) {$\displaystyle j=10$}
     \end{overpic}
  \end{minipage}
  \caption{The residuals of $10$ target eigenpairs of a $100\times 100$ real-symmetric matrix after one iteration of subspace iteration with the rational filter in~\cref{fig:exp1_setup} are plotted on the left. On the right are the eigenvector coordinates of the $1$st (circles) and $10$th (triangles) columns of $\hat Q_1$. The dangerous component (red) and the remaining target components (blue) dominate in columns $1$ and $10$, respectively. The unwanted eigenvector components (black) are filtered out almost entirely to order $u$ in the $1$st column, but are orders of magnitude larger in the $10$th column, with magnitude near $u/d$. 
\label{fig:exp1_results}}
\end{figure}
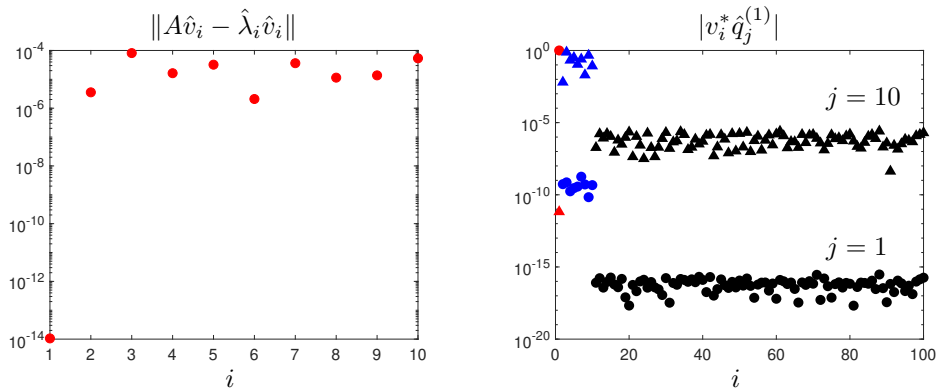

The large residuals are best explained with a look at the computed orthonormal basis $\hat Q_1$ 
(we denote computed quantities with a hat throughout, so $\hat Q_1$ is the computed approximant to $Q_1$)
in the eigenvector coordinates in~\cref{fig:exp1_results}. The first column of $\hat Q_1$ (circular markers) looks as expected: the dangerous eigenvector dominates and the unwanted components are near the unit round-off in magnitude. However, the magnitude of the unwanted components is much larger, on the order of $u/d$, in the remaining columns $\hat q_2^{(1)},\ldots,\hat q_m^{(1)}$. The $10$th column (triangular markers in~\cref{fig:exp1_results}) is representative of this observation. Although the quality of the filter means that the unwanted components should be on the order of $u$ in the columns of $\hat Q_1$, they are polluted with noise on the order of $u/d$ in all but the first column. Consequently, the accuracy in the remaining Ritz pairs computed from $\hat Q_1$ is also degraded to $u/d$.

There are two potential sources of error degrading the accuracy in $\hat Q_1$. The first is the most obvious: round-off errors are amplified when solving the ill-conditioned linear system associated with the dangerous eigenvalue. The second source is more subtle: the overwhelming dominance of the dangerous eigenvector in each column of $X_1$ leads to an ill-conditioned basis for $\mathcal{S}_1$. Remarkably, the heart of the story in~\cref{fig:arnoldi_fsi,fig:contour_refine} is contained in the latter, subtler effect and we can learn a great deal without mentioning errors incurred while applying $r(A)$. Of course, a thorough understanding requires a careful treatment of the ill-conditioned linear systems and the accumulation of errors at each iteration. We address both points in~\cref{sec:PSI_analysis}, where we study the convergence and stability of the iteration in~\cref{eqn:ratSI} when computed in floating-point arithmetic. For now, we focus on the influence of ill-conditioning in the iterates $X_1,X_2,\ldots$, noting that round-off errors in the computed iterates have little effect on their condition number (see~\cref{sec:PSI_analysis} for a full explanation).

\subsection{Accuracy of the computed orthonormal basis}\label{sec:sensitive_ONB}

When a basis $X\in\mathbb{C}^{n\times m}$ is ill-conditioned, small perturbations to the columns can have a large effect on their span. This is reflected in the sensitivity of the orthogonal factor in the QR factorization, $Q={\rm qf}(X)$. If, for some small $\epsilon>0$, $X$ is perturbed by $\Delta X$ with $\|\Delta X\|\leq \epsilon\|X\|$, then there is a $\Delta Q$ such that $Q+\Delta Q={\rm qf}(X+\Delta X)$ and~\cite[p.~382]{higham2002accuracy}
\begin{equation}\label{eqn:perturbed_QR}
\|\Delta Q\|\leq c_m\kappa(X)\|\Delta X\|/\|X\|.
\end{equation}
Here, $c_m$ is a modest constant depending only on the dimension $m$ and $\kappa(\cdot)$ denotes the $2$-norm condition number of a rectangular matrix. \Cref{eqn:perturbed_QR} tells us that when $X$ is highly ill-conditioned, the QR factorization may be extremely sensitive to perturbations. When we compute an orthonormal basis $\hat Q$ in floating-point arithmetic, we are not guaranteed accuracy much better than $\|\hat Q - Q\|\leq c_m\kappa(X)u$ (at least, as long as the columns of $X$ do not vary significantly in magnitude).

Because the rational filter amplifies the $v_1$ component in each column of $Q_0$ by $1/d$ in~\cref{eqn:ratSI}, $X_1$ is usually extremely ill-conditioned. Intuitively, $\kappa(X_1)$ cannot be much worse than $|r(\lambda_1)|/|r(\lambda_m)|$ and not much better than $|r(\lambda_1)|/|r(\lambda_2)|$ because $v_1$ is present in each column with magnitude near $|r(\lambda_1)|$ while the rest of the target eigenpairs are present with magnitude at least $|r(\lambda_m)|$ and no greater than $|r(\lambda_2)|$. \Cref{lem:X1_condition} makes this intuition precise in the form of an upper bound and asymptotic lower bound. The implication is that the error in the computed orthonormal basis $\hat Q_1$ is on the order of $u/d$ as long as the columns of $Q_0$ are not orthogonal to the dangerous eigenvector, as we observed in~\cref{fig:exp1_results}. 

We use the shorthand notation $f(x)\lesssim g(x)$ to denote the asymptotic relation
\begin{equation}
f(x)\leq g(x)(1+o(1)), \qquad\text{as}\qquad x\rightarrow 0.
\end{equation}
Note that this is slightly sharper than $f(x)=\mathcal{O}(g(x))$, but weaker than $f(x)\sim g(x)$.\footnote{This definition of $f\lesssim g$ is sharper than its common usage in the analysis of partial differential equations, where it means $f\leq Cg$ for some constant $C>0$~\cite[p.~xiv]{tao2006nonlinear}.}

\begin{proposition}
\label{lem:X1_condition}
Let normal $A\in\mathbb{C}^{n\times n}$ and $r:\Lambda\rightarrow\mathbb{C}$ satisfy~\cref{eqn:eigendecomposition,eqn:eig_index}, respectively, and given orthonormal $Q_0\in\mathbb{C}^{n\times m}$, let $X_1=r(A)Q_0$. If $V_1^*Q_0$ has full rank, then the condition number of $X_1$ satisfies
\begin{equation}\label{eqn:X1_condition}
\frac{\|v_1^*Q_0\|}{d|r(\lambda_2)|} \lesssim \kappa(X_1)\leq\Big|\frac{r(\lambda_1)}{r(\lambda_m)}\Big|\|(V_1^*Q_0)^{-1}\|, \qquad\text{as}\qquad d\rightarrow 0.
\end{equation}
\end{proposition}
\begin{proof}
The condition number of $X_1$ may be written as $\kappa(X_1)=\sigma_1(X_1)/\sigma_m(X_1)$, where $\sigma_1(X_1)\geq\cdots\geq\sigma_m(X_1)$ are the singular values of $X_1$. To bound $\sigma_1(X_1)$ above, we substitute the spectral decomposition $r(A)=Vr(\Lambda)V^*$ into the definition of $X_1$ and estimate $\sigma_1(X_1)\leq |r(\Lambda_1)|\|V^*Q_0\|\leq|r(\lambda_1)|$. To bound $\sigma_m(X_1)$ below, we use the spectral decomposition in~\cref{eqn:eigendecomposition} to write
\begin{equation}\label{eqn:spec_decomp2}
X_1 =\begin{bmatrix}
V_1 & V_2
\end{bmatrix}
\begin{bmatrix}
M_1 \\ M_2
\end{bmatrix},
\end{equation}
where $M_1=r(\Lambda_1)V_1^*Q_0$ and $M_2=r(\Lambda_2)V_2^*Q_0$. Because $V$ is unitary, the singular values of $X_1$ are precisely those of 
$\big[\begin{smallmatrix}  M_1\\M_2\end{smallmatrix}\big]$. Furthermore, $\sigma_m(X_1)\geq\sigma_m(M_1)$ since adding rows can only increase the singular values of a matrix. Finally, since $\sigma_m(M_1)=\|M_1^{-1}\|^{-1}$, we have that $\kappa(X_1)\leq \sigma_1(X_1)\|M_1^{-1}\|$. We estimate that
$$
\|M_1^{-1}\|=\|(V_1^*Q_0)^{-1}r(\Lambda_1)^{-1}\|\leq\|(V_1^*Q_0)^{-1}\||r(\lambda_m)|^{-1}.
$$
Collecting the bounds on $\sigma_1(X_1)$ and $\|M_1^{-1}\|$ establishes the upper bound in~\cref{eqn:X1_condition}.

To establish the asymptotic lower bound, we apply~\cref{eqn:amplify_v1} to $r(A)Q_0$, obtaining
\begin{equation}\label{eqn:X1_asymptotics}
X_1 = \sum_{i=1}^n r(\lambda_i)v_iv_i^*Q_0 = \frac{v_1^*Q_0}{de^{i\theta}}v_1 + \mathcal{O}(1), \quad\text{as}\quad d\rightarrow 0.
\end{equation}
Taking norms provides the asymptotic lower bound on $\sigma_1(X_1)$. To obtain a lower bound on $\sigma_m(X_1)^{-1}$, we can bound $\sigma_m(X_1)$ from above with an interlacing property for singular values of matrices subject to rank one perturbations. We rewrite~\cref{eqn:X1_asymptotics} as
$$
X_1 = r(\lambda_1)v_1v_1^*Q_0 + V{\rm diag}(0, r(\lambda_2), \ldots, r(\lambda_n))V^*Q_0= N_1 + N_2.
$$
Now, $\sigma_2(N_1)=0$ and $\sigma_1(N_2)\leq|r(\lambda_2)|$ so, by interlacing~\cite{thompson1976behavior}, we obtain the estimate
$$
\sigma_2(X_1)\leq \sigma_1(N_2) + \sigma_2(N_1)\leq|r(\lambda_2)|.
$$
As $\sigma_m(X_1)\leq \sigma_2(X_1)$ implies that $1/\sigma_m(X_1)\geq |r(\lambda_2)|^{-1}$, collecting lower bounds concludes the proof of~\cref{eqn:X1_condition}.
\end{proof}

The factor $\|(V_1^*Q_0)^{-1}\|$ in~\cref{lem:X1_condition} appears naturally in connection with subspace iteration, and we will encounter it again in~\cref{sec:PSI_analysis}. It is precisely the reciprocal of $\cos\theta_1(\mathcal{S}_k,\mathcal{V})$ (see~\cref{def:PABS}), approaching unity when $\mathcal{V}$ and $\mathcal{S}_0$ are nearby and blowing up quadratically when they are made orthogonal. In~\cref{lem:X1_condition} it indicates that $X_1$ may suffer additional ill-conditioning if the initial subspace $\mathcal{S}_0$ is accidentally chosen to be too near orthogonal to $\mathcal{V}$.\footnote{When $Q_0$ is selected so that its entries are independent, identically distributed Gaussian random variables, $\|(V_1^*Q_0)^{-1}\|$ is roughly $\sqrt{m}$ in expectation, but can be an order of magnitude or so larger with nontrivial probability. A powerful workaround is to work with a slightly larger subspace and take $m$ larger than the number of target eigenvalues; this dramatically reduces the probability of large $\|(V_1^*Q_0)^{-1}\|$~\cite{davidson2001local}.}

\section{Twice is enough}\label{sec:twice_is_enough}

In~\cref{lem:X1_condition}, the asymptotic lower bound in~\cref{eqn:X1_condition} plummets if the columns of $Q_0$ are taken nearly orthogonal to $v_1$, the dangerous eigenvector. This is because the rational filter has nothing to amplify when $v_1$ is absent in the columns of $Q_0$. If $v_1$ is present with magnitude no greater than $\mathcal{O}(d)$ in $Q_0$, then the columns of $X_1$ are not strongly aligned along any single eigenvector and the conditioning of $X_1$ is likely to improve. Crucially, this intuition holds even if $v_1$ dominates in one column but not the others. The main point is that the columns of $X_1$ are no longer necessarily close to a linearly dependent set.

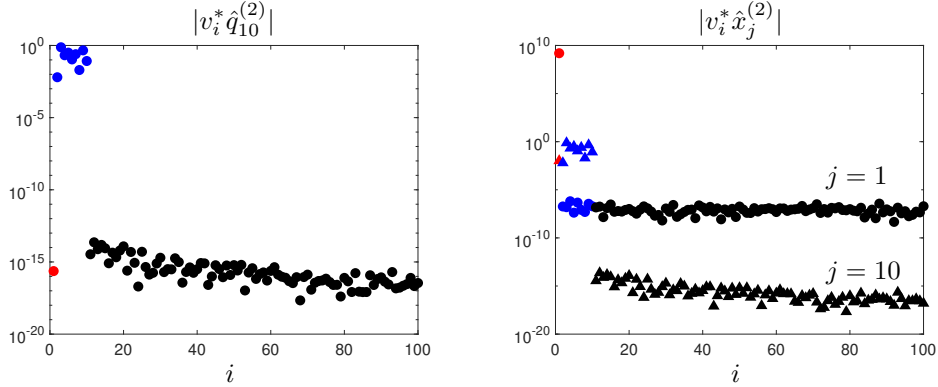
\begin{figure}[!tbp]
  \centering
  \begin{minipage}[b]{0.48\textwidth}
    \begin{overpic}[width=\textwidth]{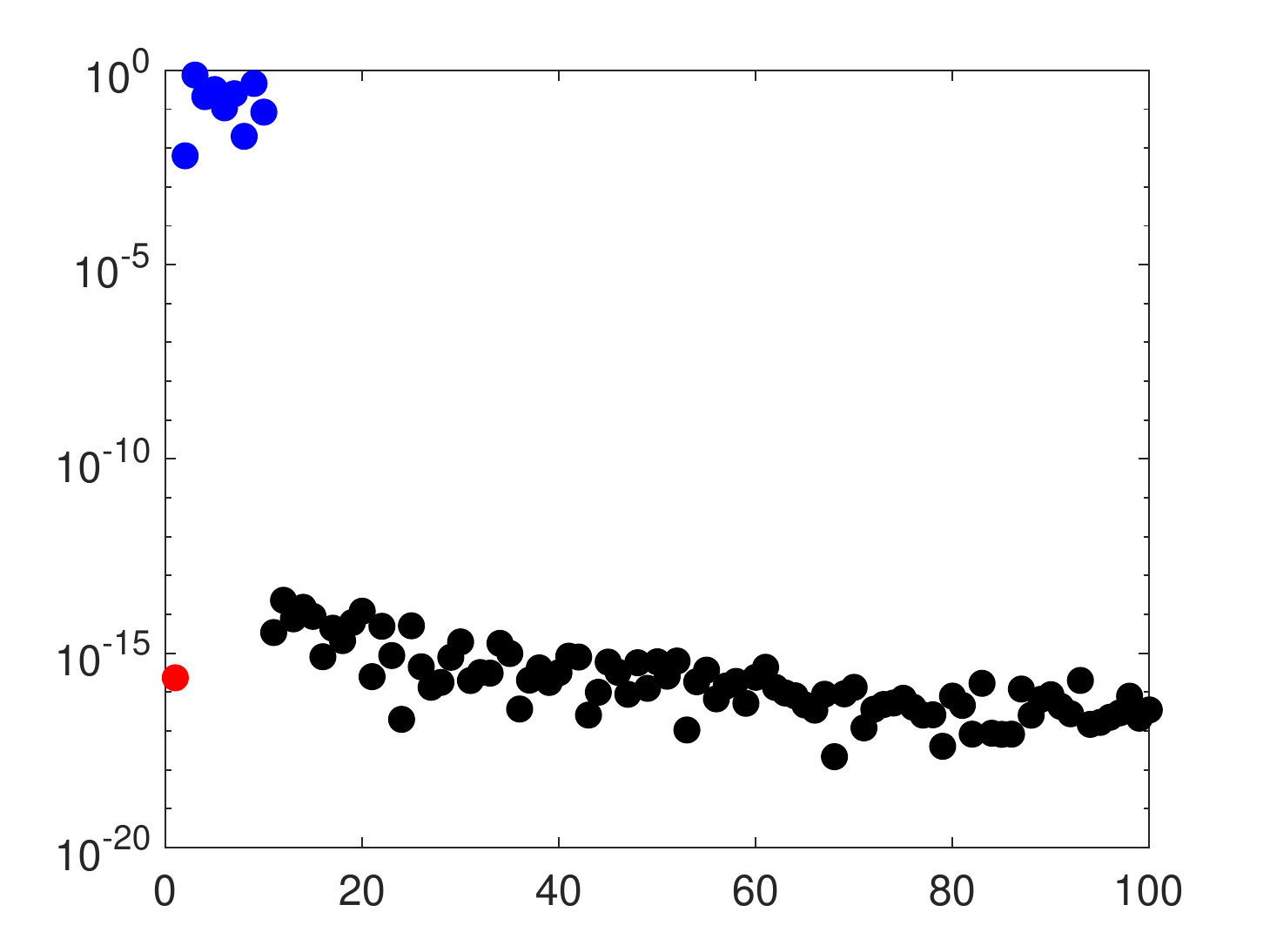}
     \put (43,73) {$\displaystyle |v_i^*\hat q^{(2)}_{10}|$}
     \put (50,-2) {$\displaystyle i$}
     \end{overpic}
  \end{minipage}
  \hfill
  \begin{minipage}[b]{0.48\textwidth}
    \begin{overpic}[width=\textwidth]{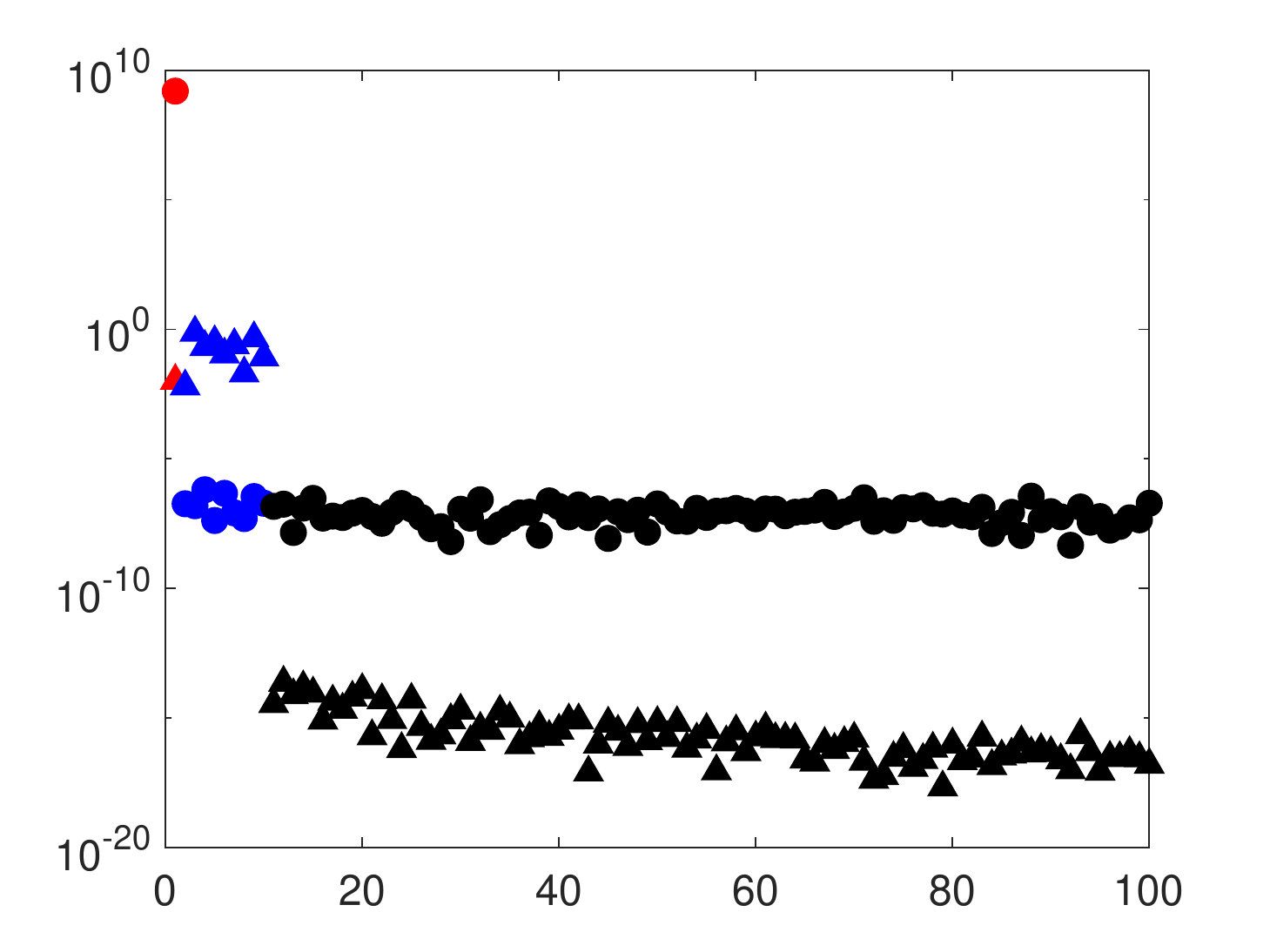}
     \put (43,73) {$\displaystyle |v_i^*\hat x^{(2)}_j|$}
     \put (50,-2) {$\displaystyle i$}
   	 \put (70,40) {$\displaystyle j=1$}
   	 \put (70,20) {$\displaystyle j=10$}
     \end{overpic}
  \end{minipage}
  \caption{The structure of the iterates $\hat X_2$ and $\hat Q_2$ after the second iteration of subspace iteration with a rational filter. On the left, the eigenvector coordinates of the $10$th column of the computed orthonormal basis color-coded for dangerous component (red), remaining target components (blue), and unwanted components (black). On the right, the eigenvector coordinates of the $1$st (circles) and $10$th (triangles) columns of the computed basis $\hat X_2$ with the same color code used in the left panel.
\label{fig:exp2_results}}
\end{figure}

Let us return to the example of~\cref{fig:exp1_setup}. If we print out the residual norms of the target eigenpairs after the second iteration of~\cref{eqn:ratSI}, we see remarkable improvement:
\begin{verbatim}
   	 6.7997e-15   2.5942e-14    2.268e-13   4.3433e-14   9.1978e-14
   	 1.3716e-14   9.7045e-14   3.4121e-14   1.4594e-13   4.0235e-14
\end{verbatim}
Now all the target pairs have been resolved to within $13$ or $14$ digits of accuracy, in contrast to~\cref{fig:exp1_results} (left). If we examine the computed orthonormal basis used to extract the Ritz pairs, we observe that the noise in the direction of the unwanted eigenvectors has also been reduced to the order of $u$, compared with $u/d$ in the first iteration. \Cref{fig:exp2_results} (left) illustrates the composition of the $10$th column of $\hat Q_2$, which is representative of the last $m-1$ columns.
 
The reason for the restored accuracy in the computed orthonormal basis is that, unlike $X_1$, the basis $X_2$ has an even blend of the target eigenvector directions in all but the first of its columns. \Cref{fig:exp2_results} (right) displays the magnitude of the eigenvector coordinates for the first (circular markers) and last (triangular markers) columns of the computed basis, $\hat X_2$. In the first column, the dangerous direction is effectively the only direction present, since all other components appear with relative magnitude near $u$. In contrast, the last column of $\hat X_2$ contains order one components in each target direction with the unwanted directions completely filtered out. The remaining columns of $\hat X_2$ are similar in composition to the last. Without $v_1$ dominating in every column, we can accurately extract an orthonormal basis.

The clue to the stark difference in the composition of $\hat X_1$ and $\hat X_2$ is contained in~\cref{fig:exp1_results}. We see that the first column of $\hat Q_1$ is dominated by the dangerous eigenvector, up to the $9$th or $10$th digit. Consequently, the remaining columns of $\hat Q_1$ are nearly orthogonal to $v_1$. We observe this in~\cref{fig:exp1_results} (right), where $v_1^*q_{10}^{(1)}\approx 10^{-11}$. When the rational filter is applied to $\hat Q_1$ in the second iteration, the amplification of $v_1$ restores an even blend of the target eigenvectors in the last $m-1$ columns of $\hat X_2$, rather than boosting $v_1$ above the others.

\subsection{A well-conditioned basis}\label{sec:well_conditioned_basis}

Motivated by the preceding discussion, we now examine the condition number of the second iterate, $X_2$, when $q_1^{(1)}$ is a good approximation to $v_1$ and, consequently, all but one of the columns of $Q_1$ are deficient in the dangerous direction. We then briefly explain the structure of the eigenvector coordinates for the computed orthonormal basis $\hat Q_1$ observed in~\cref{fig:exp1_results}.

To investigate how a weak presence of $v_1$ in columns of $Q_1$ improves the conditioning of $X_2$, we break the target eigenvector coordinates of $Q_1$ into blocks, as
\begin{equation}\label{eqn:Q1_structure}
V_1^*Q_1 = \begin{bmatrix}
v_1^*q_1^{(1)} & v_1^*\tilde Q_1 \\
\tilde V_1^*q_1^{(1)} & \tilde V_1^*\tilde Q_1
\end{bmatrix}
=\begin{bmatrix}
a & b \\
c & D
\end{bmatrix}.
\end{equation}
Here, we use $\tilde V_1$ and $\tilde Q_1$ to denote the $n\times (m-1)$ matrices formed by removing the first columns of $V_1$ and $Q_1$, respectively. When $q_1^{(1)}$ closely approximates $v_1$, $\|c\|$ is small due to the orthogonality of the eigenvectors. Moreover, $\|b\|$ is also small because the columns of $\tilde Q_1$ are nearly orthogonal to $v_1$. Let us suppose that $q_1^{(1)}=v_1+\mathcal{O}(d)$, so that $b$ and $c$ are $\mathcal{O}(d)$ (we explain why this holds momentarily, following~\cref{thm:twice-is-enough}).

After applying the rational filter to $Q_1$, the eigenvector coordinates $V_1^*X_2$ inherit a natural block structure from~\cref{eqn:Q1_structure}. Letting $\tilde\Lambda_1={\rm diag}(\lambda_2,\ldots,\lambda_m)$, we have that
\begin{equation}\label{eqn:X2_structure}
V_1^*X_2 = r(\Lambda_1)V_1^*Q_1
=\begin{bmatrix}
r(\lambda_1)a & r(\lambda_1)b \\
r(\tilde\Lambda_1)c & r(\tilde\Lambda_1)D
\end{bmatrix}.
\end{equation}
While the bottom left block remains small, with norm no greater than $|r(\lambda_2)|\|c\|=\mathcal{O}(d)$, the entire first row is amplified by $|r(\lambda_1)|$ so that $\|r(\lambda_1)b\|\approx \|b\|/d$.

To estimate the condition number of $X_2$, we scale the columns of $X_2$ with the $m\times m$ diagonal matrix
\begin{equation}\label{eqn:column_scaling}
T={\rm diag}(r(\lambda_1)^{-1},1,\ldots,1).
\end{equation}
This diagonal scaling does not alter ${\rm span}(X_2)$ or the sensitivity of the orthonormal basis, $Q_2={\rm qf}(X_2)$. (Note that scaling the columns of $X_1$ has no effect in~\cref{lem:X1_condition}, as the columns of $X_1$ all have magnitude $\approx 1/d$.) However, it conveniently puts the diagonal blocks in~\cref{eqn:X2_structure} on equal footing and ensures that $\sigma_1(X_2T)=\mathcal{O}(1)$, so that any ill-conditioning due to the dangerous eigenvalue is captured in the smallest singular value of $X_2T$. This allows us to focus on computing a lower bound for $\sigma_m(X_2T)$ or, equivalently, an upper bound for $1/\sigma_m(X_2T)$. Just as in the proof of~\cref{lem:X1_condition}, it suffices to bound $\|(V_1^*X_2T)^{-1}\|$ from above (assuming as usual that $V_1^*Q_1$, and therefore $V_1^*X_2$, has full rank).

With all of the ingredients in place, the estimate is fairly straightforward. After the column scaling, $V_1^*X_2T$ is approximately block upper triangular. We have that
\begin{equation}\label{eqn:block_UT}
V_1^*X_2T
=\begin{bmatrix}
r(\lambda_1)a & r(\lambda_1)b \\
r(\tilde\Lambda_1)c & r(\tilde\Lambda_1)D
\end{bmatrix}T
=
\begin{bmatrix}
a & b/(de^{i\theta}) \\
 & r(\tilde\Lambda_1)D
\end{bmatrix}
+
\mathcal{O}(d).
\end{equation}
We can apply the formula for $2\times 2$ block upper triangular matrix inversion and the fact that matrix inversion is locally Lipschitz continuous to compute (for $d$ sufficiently small)
\begin{equation}\label{eqn:UT_inverse}
(V_1^*X_2T)^{-1}
=
\begin{bmatrix}
a^{-1} & -a^{-1}D^{-1}r(\tilde\Lambda_1)^{-1}b/(de^{i\theta}) \\
 & D^{-1}r(\tilde\Lambda_1)^{-1}
\end{bmatrix}\left(I+\mathcal{O}(d)\right).
\end{equation}
The norm of the block upper triangular matrix in~\cref{eqn:UT_inverse} is bounded by the sum of the norms of the blocks, so we conclude that $1/\sigma_m(X_2T)\leq \|(V_1^*X_2T)^{-1}\|=\mathcal{O}(1)$ when $\|b\|=\mathcal{O}(d)$. Estimating the norms of these blocks individually and combining with an estimate for $\sigma_1(X_2T)$ leads to the following upper bound on $\kappa(X_2T)$.

\begin{theorem}[Twice-is-enough]\label{thm:twice-is-enough}
Let normal $A\in\mathbb{C}^{n\times n}$ and $r:\Lambda\rightarrow\mathbb{C}$ satisfy~\cref{eqn:eigendecomposition,eqn:eig_index}, respectively, and given orthonormal $Q_1\in\mathbb{C}^{n\times m}$, let $X_2=r(A)Q_1$. Let $b$ and $D$ denote the blocks of $V_1^*Q_1$ in~\cref{eqn:Q1_structure}. If $D$ is invertible and the first column of $Q_1$ satisfies $\|q_1^{(1)}-v_1\|=\mathcal{O}(d)$, then $\|b\|=\mathcal{O}(d)$ and
\begin{equation}\label{eqn:twice-is-enough}
\kappa(X_2T)\leq M\left(\left(\frac{\|b\|}{d}+1\right)\frac{\|D^{-1}\|}{|r(\lambda_m)|}+1\right)+\mathcal{O}(d) \qquad\text{as}\qquad d\rightarrow 0.
\end{equation}
Here, $T={\rm diag}(r(\lambda_1)^{-1},1,\ldots,1)\in\mathbb{C}^{m\times m}$ and $M=\|b\|/d+{\rm max}\{1,|r(\lambda_2)|\}$.
\end{theorem}
\begin{proof}
First, the hypothesis $\|q_1^{(1)}-v_1\|=\mathcal{O}(d)$ immediately implies that $|a|=1+\mathcal{O}(d)$ and $\|b\|=\mathcal{O}(d)$. Then, following the discussion above, it suffices to bound $\|X_2T\|$ and the norms of the blocks in~\cref{eqn:UT_inverse}. The condition number of $X_2T$ is bounded above by the product of these two estimates. Since $\|r(\tilde\Lambda_1)^{-1}\|=|r(\lambda_m)|^{-1}$, we obtain that
\begin{equation}\label{eqn:bound_sigmaM}
1/\sigma_m(X_2T)\leq \|(V_1^*X_2T)^{-1}\|\leq 1+\frac{\|D^{-1}\|}{|r(\lambda_m)|}\left(1+\frac{\|b\|}{d}\right)+\mathcal{O}(d).
\end{equation}
On the other hand, we can write $V^*X_2T$ in block form analogous to~\cref{eqn:block_UT}, as
$$
V^*X_2T =
\begin{bmatrix}
a & \tilde b/(de^{i\theta}) \\
 & r(\tilde\Lambda)\tilde D
\end{bmatrix}
+
\mathcal{O}(d),
$$
where $\tilde b=V^*q_1^{(1)}$ and $\tilde D=V^*\tilde Q_1$. Calculating the norm of the block diagonal component and the off-diagonal component separately and applying the triangle inequality yields $\|X_2T\|\leq M$. Collecting with the bound in~\cref{eqn:bound_sigmaM} concludes the proof.
\end{proof}

\Cref{thm:twice-is-enough} tells us that $X_2$ is only a simple column scaling away from a well-conditioned basis when the first column of $Q_1$ approximates $v_1$ with accuracy $\mathcal{O}(d)$. Since the sensitivity (and numerical computation) of the QR factorization is not effected by column scaling, the $\mathcal{O}(1)$ bound on $\kappa(X_2T)$ in~\cref{eqn:twice-is-enough} explains why the computed orthonormal basis for $\mathcal{S}_2$ is accurate to unit round-off. This line of analysis follows naturally from our observation about the eigenvector coordinates of $\hat Q_1$ in~\cref{fig:exp1_results} (right), but one question remains. Why is the first column of the computed orthonormal basis such a good approximation to $v_1$?

The answer is that $\hat q_1^{(1)}$ is essentially the first column of $X_1$ after normalization, up to the unit round-off $u$. In particular, $\hat q_1^{(1)}$ is unaffected by the $u/d$ errors in $\hat Q_1$ caused by ill-conditioning in $X_1$ (see~\cref{lem:X1_condition}). These errors are concentrated in the later columns of $\hat Q_1$ because of the nested structure of Householder reflections (or Givens rotations) used to make $X_1$ upper triangular. We have that $x_1/\|x_1\|=v_1+\mathcal{O}(d)$ by~\cref{eqn:amplify_v1}, so we expect that $\hat q_1^{(1)}=v_1+\mathcal{O}(d)$ also, as observed in~\cref{fig:exp1_results}. 

Finally, if the orthogonal factor is computed with modified Gram-Schmidt instead of Householder reflections or Givens rotations, the columns of $\hat Q_1$ lose orthogonality in proportion to the condition number of the ill-conditioned basis $X_1$. The consequence of this is that the block $v_1^*(\hat Q_1)_{(2:m)}$ from~\cref{eqn:Q1_structure} may be as large as $u/d$ instead of $\mathcal{O}(d)$, even though $\hat q_1^{(1)}=v_1+\mathcal{O}(d)$. This may alter the order of magnitude of $\kappa(X_2T)$ when $d\ll\sqrt{u}$ (since then, $u/d\gg d)$ as the balance in~\cref{thm:twice-is-enough} is disrupted. In particular, twice may no longer be enough to correct ill-conditioning in $\hat X_2$. A similar effect is observed for non-normal matrices in~\cref{sec:non-normal_case} even when Householder reflections or Givens rotations are employed in the QR factorizations.

\section{Convergence and stability}\label{sec:PSI_analysis}

So far, our analysis of dangerous eigenvalues has focused on the conditioning of the iterates $X_1,X_2,\ldots$ in~\cref{eqn:ratSI} and the corresponding accuracy in the computed orthonormal bases. Indeed, this perspective explains the $u/d$ errors observed in the first iteration (see~\cref{fig:exp1_results}) and provides the essential insight into the restored accuracy observed in the second iteration (see~\cref{fig:exp2_results}). But we have not yet explained how the round-off errors incurred while applying the ill-conditioned rational filter enter the picture. Nor have we discussed how these round-off errors, together with the error in the computed orthonormal basis, accumulate during the iterations in~\cref{eqn:ratSI}.

To apply the rational filter $r(A)$ to an $n\times m$ matrix $Q$ in practice, one solves linear systems with a shift at each pole and takes a weighted average of the solutions:
\begin{equation}\label{eqn:filter2solves}
r(A)Q=\sum_{j=1}^\ell \omega_jX^{(j)}, \quad\text{where}\quad (z_jI-A)X^{(j)}=Q, \quad j=1,\ldots,\ell.
\end{equation}
If the linear systems are solved with a backward stable algorithm, then the computed solutions $\hat X^{(j)}$ satisfy, for each $j=1,\ldots,\ell$,
\begin{equation}\label{eqn:back_error}
(z_jI-A-\mathcal{E}_j)\hat X^{(j)}=Q, \qquad \|\mathcal{E}_j\|\leq\gamma\|A\| u.
\end{equation}
Here, $\mathcal{E}_j$ is the backward error and $\gamma$ is a constant, with modest dependence on $z_1,\ldots,z_\ell$ and the dimension of $Q$, such that $\gamma u\ll 1$ for typical situations.\footnote{This characterization can be modified to accommodate inexact solution techniques, such as iterative methods, but $\gamma$ may be much larger, depending on the stability properties of the particular numerical method~\cite[p. 339]{stewart2001matrix}.} 

Now, if we neglect errors made while forming the linear combination on the left-hand side of~\cref{eqn:filter2solves}, then the forward error in $r(A)Q$ can be written as\footnote{For expositional clarity, we neglect round-off errors accrued when forming the linear combination in the right-hand side of~\cref{eqn:filter2solves} to focus on the effect of the ill-conditioned linear systems. For typical choices of the weights and nodes in $r(A)$, this amounts to discarding a term on the order of $u$ relative to the largest column norm of the $\hat X^{(j)}$.}
\begin{equation}\label{eqn:forward_filter_error}
\hat X - r(A)Q = \sum_{k=1}^\ell \omega_j(z_jI-A)^{-1}\mathcal{E}_j\hat X^{(j)}, \qquad\text{where}\qquad \hat X=\sum_{k=1}^\ell \omega_j\hat X^{(j)}.
\end{equation}
Due to the appearance of $\mathcal{E}_j$, the terms in the left-hand sum are all on the order of $u$ except for the term corresponding to the pole near the dangerous eigenvalue, whose index we call $j=j_*$. In the dangerous term, $(z_{j_*}I-A)^{-1}$ amplifies the $v_1$ components in the columns of $\mathcal{E}_{j_*}$ by a factor of $1/d$. Similarly, the components of $v_1$ in the columns of $Q$ are amplified to order $1/d$ in the corresponding columns of $\hat X^{(j_*)}$ (this is made precise by expanding~\cref{eqn:back_error} in a Neumann series). Therefore, the relative errors in the columns of $\hat X$ are on the order of $u/d$.

Thus, every time $r(A)$ is applied in~\cref{eqn:ratSI}, relative errors of order $u/d$ are accrued in the columns of $\hat X_k$. On the one hand, our understanding of accuracy in the computed orthonormal basis $\hat Q_k$ (developed in~\cref{sec:dang_eigvals,sec:twice_is_enough}) remains intact, because perturbations of relative order $u/d$ to the columns of $X_k$ have little effect on the leading-order estimates for $\kappa(X_k)$. On the other hand, we may wonder: what effect do such perturbations have on ${\rm span}(X_k)$ and the geometric convergence implied in~\cref{thm:1step_standard}? 

Recent analyses of subspace iteration accelerated with a rational filter suggest that ${\rm span}(\hat X_k)$ tends to $\mathcal{V}$ geometrically at roughly the expected rate until a threshold of accuracy is reached, at which point convergence plateaus~\cite{peter2014feast}. This threshold is usually the same order of magnitude as the error accrued in the subspace at each iteration, i.e., in the columns of $\hat X_k$. Similar results have been derived for perturbations in the entries of the matrix $r(A)$ (this work does not consider filters explicitly)~\cite{saad2016analysis}. However, the evidence of the experiments in~\cref{fig:arnoldi_fsi,fig:contour_refine} and in~\cref{sec:twice_is_enough} indicates that errors in ${\rm span}(\hat X_k)$ caused by dangerous eigenvalues do not prevent the Rayleigh--Ritz procedure from finding vectors in ${\rm span}(\hat X_k)$ that approximate the target eigenvectors to unit round-off accuracy. We now show that errors in $\hat X_k$ caused by the dangerous eigenvalue do not lead to early stagnation or instability in the computed iterates. In the worst case, they may slow the geometric convergence rate by a factor of roughly $(1-u/d)^{-1}$. Moreover, the iteration is stable as long as the columns of the initial guess $Q_0$ are not too near to $\mathcal{V}^\perp$ (see~\cref{fig:stability}).

\subsection{One-step refinement bounds}\label{sec:perturbed_refinement}

The amplifying power of the dangerous eigenvalue leads to large relative errors in the columns of $\hat X_k$. However, the errors possess an important quality: the amplification is entirely in the direction of $v_1$ so that the relative errors in the unwanted direction are still small. To understand how these structured perturbations influence $\mathcal{\hat S}_k={\rm span}(\hat X_k)$, we gather the errors accrued during the $k$th iteration into a perturbation to the orthonormal basis for $\mathcal{\hat S}_{k-1}$ and construct a one-step refinement bound as in~\cref{thm:standardSI_bound}. Formulated precisely, we replace~\cref{eqn:ratSI} with the perturbed form
\begin{equation}\label{eqn:ratSI_perturbed}
\hat X_k=r(A)(Q_{k-1}'+R_k), \qquad Q_k'={\rm qf}(\hat X_k).
\end{equation}
Note that we include any errors in the computed orthonormal factor in $R_k$, placing the emphasis on $\mathcal{\hat S}_k={\rm span}(\hat X_k)={\rm span}(Q_k')$ rather than on ${\rm span}(\hat Q_k)$. This causes no difficulty since, as we know from~\cref{sec:twice_is_enough}, the error $\hat Q_k-Q_k'$ is on the order of $u$ for $k\geq 2$. Since $Q_k'$ is an orthonormal basis, $\mathcal{\hat S}_k$ and ${\rm span}(Q_k')$ only differ by a term not much larger than $u$.

To begin, we establish the form~\cref{eqn:ratSI_perturbed} by way of the residuals of the linear systems in~\cref{eqn:forward_filter_error} and study the structure of $R_k$. To measure the columns of $R_k$ relative to the columns of $\hat X_k$, it is convenient to apply the diagonal scaling
\begin{equation}\label{eqn:diag_scale}
C_k={\rm diag}(\|(\hat X_k)_i\|^{-1},\ldots,\|(\hat X_k)_m\|^{-1})/\sqrt{m},
\end{equation}
so that $\|\hat X_kC_k\|\leq 1$. We also need the majorization of the rational filter, denoted
\begin{equation}\label{eqn:mod_filter}
\tilde r(\lambda)=\sum_{j=1}^\ell |\omega_j||(z_j-\lambda)^{-1}|.
\end{equation}
As usual, $\tilde r(\Lambda_1)$ and $\tilde r(\Lambda_2)$ are the matrices when the function in~\cref{eqn:mod_filter} is applied to the diagonal matrices $\Lambda_1$ and $\Lambda_2$. Observe that $\|\tilde r(\Lambda_1)r(\Lambda_1)^{-1}\|=\mathcal{O}(1)$ as $d\rightarrow 0$, because the poles near the dangerous eigenvalue cancel.

\begin{lemma}\label{lem:filter_FBsplit}
Let normal $A\in\mathbb{C}^{n\times n}$ and $r:\Lambda\rightarrow\mathbb{C}$ satisfy~\cref{eqn:eigendecomposition} and~\cref{eqn:eig_index}, respectively. Given $Q\in\mathbb{C}^{n\times m}$, let $\hat X=\sum_{j=1}^\ell \omega_j\hat X^{(j)}$ with each $\hat X^{(j)}$ satisfying~\cref{eqn:back_error}. Then, there is an $R\in\mathbb{C}^{n\times m}$ such that $\hat X=r(A)(Q+R)$ and
\begin{equation}\label{eqn:filter_FBsplit}
\|P_\mathcal{V}R\|\leq\gamma_1\|A\| u/d \qquad\text{and}\qquad \|(I-P_\mathcal{V})r(A)RC\|\leq\gamma_2\|A\| u.
\end{equation}
Here, $\gamma_1=\gamma\,\|r(\Lambda_1)^{-1}\tilde r(\Lambda_1)\|$, $\gamma_2=\gamma\,\|\tilde r(\Lambda_2)\|$, and $C$ is the diagonal scaling in~\cref{eqn:diag_scale} (with index $k$ suppressed).
\end{lemma}
\begin{proof}
Because each $\hat X^{(j)}$ satisfies~\cref{eqn:back_error} and $\hat X=\sum_{j=1}^\ell\omega_j\hat X^{(j)}$, we collect like terms in~\eqref{eqn:forward_filter_error} and compute
\begin{equation}\label{eqn:multiple_poles_fp}
\hat X = r(A)Q + \sum_{j=1}^\ell \omega_j(z_jI-A)^{-1}R^{(j)},
\end{equation}
where $R^{(j)}=\mathcal{E}^{(j)}\hat X$. Note that $\|R^{(j)}C\|\leq\gamma\|A\|u$, for $j=1,\ldots,\ell$, by~\cref{eqn:back_error}. 

We compute $R$ directly by comparing~\cref{eqn:multiple_poles_fp} with $\hat X=r(A)(Q+R)$ and noting that we need $r(A)R = \sum_{j=1}^\ell \omega_j(z_jI-A)^{-1}R^{(j)}$. Inserting the eigenvalue decomposition $A=V\Lambda V^*$ into both sides and inverting $r(A)=Vr(\Lambda)V^*$, we obtain
\begin{equation}\label{eqn:filter_backErr}
R = Vr(\Lambda)^{-1}\left(\sum_{j=1}^\ell \omega_j(z_jI-\Lambda)^{-1}V^*R^{(j)}\right).
\end{equation}
Calculating $P_\mathcal{V}R$ and $(I-P_\mathcal{V})r(A)RC$ directly from~\cref{eqn:filter_backErr} and applying the backward error bounds in~\cref{eqn:back_error} to bound the residuals $\|R^{(j)}\|$ uniformly, we obtain the bounds in~\cref{eqn:filter_FBsplit}.
\end{proof}

\Cref{lem:filter_FBsplit} demonstrates that the perturbations $R_k$ in~\cref{eqn:ratSI_perturbed} capture the essential structure of the errors in $\hat X_k$. First, $R_k$ perturbs $Q_{k-1}'$ with relative magnitude $u/d$ and direction in the subspace $\mathcal{V}$. Second, $r(A)R_k$ perturbs the columns of $X_k$ with relative magnitude $u$ and direction in the subspace $\mathcal{V}^\perp$. We note that $\|V_2^*R_kC_k\|$ itself is not small when the filter is very good, i.e., close to unit-round off on the unwanted eigenvalues, as $\|r(\Lambda_2)^{-1}\tilde r(\Lambda_2)\|$ may be extremely large. However, the forward application of the filter cancels any large factors in $r(\Lambda_2)^{-1}$ exactly.

With~\cref{lem:filter_FBsplit} in hand, we can calculate a one-step refinement bound generalizing~\cref{thm:standardSI_bound} to the perturbed iteration in~\cref{eqn:ratSI_perturbed}. While the $u/d$ relative errors in $\hat X_k$ are felt in the refinement factor in~\cref{eqn:1step_perturbed}, they do not appear in the additive perturbation to $\tan\theta_1(\mathcal{\hat S}_k,\mathcal{V})$. This point is crucial because, as we show in~\cref{sec:no_stagnation}, the size of the additive term determines the threshold for stagnation in the worst\rr{-}case accumulation of errors.
\begin{theorem}\label{thm:1step_perturbed}
Let normal $A\in\mathbb{C}^{n\times n}$ and $r:\Lambda\rightarrow\mathbb{C}$ satisfy~\cref{eqn:eigendecomposition} and~\cref{eqn:eig_index}, respectively, and let $\mathcal{\hat S}_k={\rm span}(\hat X_k)$, with $\hat X_k$ defined in~\cref{eqn:ratSI_perturbed} and $R_k$ satisfying~\cref{eqn:filter_FBsplit}. If $\cos\theta_1(\mathcal{\hat S}_{k-1},\mathcal{V}) > \gamma_1\|A\|u/d$ and $\cos\theta_1(\mathcal{\hat S}_k,\mathcal{V})>0$, then
\begin{equation}\label{eqn:1step_perturbed}
\tan\theta_1(\mathcal{\hat S}_k,\mathcal{V})\leq\Big|\frac{r(\lambda_{m+1})}{r(\lambda_m)}\Big|\frac{\tan\theta_1(\mathcal{\hat S}_{k-1},\mathcal{V})}{1-\alpha_k}+\beta_k,
\end{equation}
where $\alpha_k\leq\gamma_1\|A\|u/(d\cos\theta_1(\mathcal{\hat S}_{k-1},\mathcal{V}))$ and $\beta_k\leq\gamma_2\|A\|\kappa(\hat X_kC_k)u/\cos\theta_1(\mathcal{\hat S}_k,\mathcal{V})$.
\end{theorem}
\begin{proof}
Calculating directly as in the proof of~\cref{thm:standardSI_bound}, we have that
\begin{equation}\label{eqn:perturbed_tan}
T(\hat X_kC_k,V_1)=(I-P_\mathcal{V})r(A)(Q_{k-1}'+R_k)C_k(V_1^*\hat X_kC_k)^{-1}.
\end{equation}
We proceed by bounding the two terms in~\cref{eqn:perturbed_tan} corresponding to $Q_{k-1}'$ and $R_k$. By~\cref{lem:filter_FBsplit}, $\|(I-P_\mathcal{V})r(A)R_kC_k\|\leq\gamma_2\|A\|u$. If $\hat X_kC_k=Q_k'S_k$ is an economy-sized QR factorization, then the singular values of $S_k$ and $\hat X_kC_k$ coincide, and
\begin{equation}
\|(V_1^*\hat X_kC_k)^{-1}\|=\|S_k^{-1}(V_1^*Q_k')^{-1}\|\leq\sigma_m(\hat X_kC_k)/\cos\theta_1(\mathcal{\hat S}_k,\mathcal{V}).
\end{equation}
Since $\|\hat X_kC_k\|\leq 1$, we conclude that $\|(I-P_\mathcal{V})r(A)R_kC_k(V_1^*\hat X_kC_k)^{-1}\|\leq\beta_k$.

Now, rewrite $(V_1^*\hat X_k)^{-1}=(V_1^*(Q_{k-1}'+R_k))^{-1}r(\Lambda_1)^{-1}$. Invoking the hypothesis on $\cos\theta_1(\mathcal{\hat S}_{k-1},\mathcal{V})$, which implies that $\|(V_1^*Q_{k-1}')^{-1}V_1^*R_k\|<1$, we expand
$$
(V_1^*(Q_{k-1}'+R_k))^{-1}=\left(I+\sum_{k=1}^\infty (V_1^*Q_{k-1}')^{-k}(V_1^*R_k)^k\right)(V_1^*Q_{k-1}')^{-1}.
$$
Since $T(Q_{k-1}',V_1)=(I-P_\mathcal{V})Q_{k-1}'(V_1^*Q_{k-1}')^{-1}$, we calculate that
$$
\begin{aligned}
(I-P_\mathcal{V})r(A)Q_{k-1}'(V_1^*\hat X_k)^{-1} = r(A)(I-P_\mathcal{V})Q_{k-1}'(V_1^*(Q_{k-1}'+R_k))^{-1}r(\Lambda_1)^{-1} \\
=r(A)T(Q_{k-1}',V_1)\left(I+\sum_{k=1}^\infty (V_1^*Q_{k-1}')^{-k+1}(V_1^*R_k)^k(V_1^*Q_{k-1}')^{-1}\right)r(\Lambda_1)^{-1}.
\end{aligned}
$$
Taking norms, applying the bound for $\|P_\mathcal{V}R_k\|=\|V_1^*R_k\|$ from~\cref{lem:filter_FBsplit} along with the identity $\|(V_1^*Q_{k-1}')^{-1}\|=1/\cos\theta_1(\mathcal{\hat S}_{k-1},\mathcal{V})$ from~\cref{def:PABS}, and summing the resulting geometric series, we obtain the bound
$$
\|(I-P_\mathcal{V})r(A)Q_{k-1}'(V_1^*\hat X_k)^{-1}\|\leq \|r(\Lambda_2)\|\|r(\Lambda_1)^{-1}\|\frac{\tan\theta_1(\mathcal{\hat S}_{k-1},\mathcal{V})}{1-\alpha_k}.
$$
Noting that $\|r(\Lambda_2)\|=|r(\lambda_{m+1})|$, $\|r(\Lambda_1)^{-1}\|=|r(\lambda_m)|^{-1}$, and collecting the bounds for the two terms in~\cref{eqn:perturbed_tan} establishes~\cref{eqn:1step_perturbed}.
\end{proof}

The significance of~\cref{thm:1step_perturbed} is that the errors in $\hat X_k$ that lie in the target subspace $\mathcal{V}$ impact only the refinement rate, and do not contribute to the additive term $\beta_k$ in~\cref{eqn:1step_perturbed}. This worst\rr{-}case scenario occurs over one iteration only when the perturbations are aligned to maximally cancel the components of $\mathcal{V}$ present in the basis $Q_{k-1}'$. In fact, such errors are just as likely to align perfectly with the $\mathcal{V}$ components of $Q_{k-1}'$ and improve the refinement rate by $(1+\alpha_k)^{-1}$, so the impact on the geometric convergence rate implied by~\cref{eqn:1step_perturbed} is probably not observed in practice. 

On the other hand, the errors in $\hat X_k$ that lie in $\mathcal{V}^\perp$ degrade the expected refinement through the additive term $\beta_k$ and, due to orthogonality, have a tangible effect on the convergence of subspace iteration in floating-point arithmetic. Note that the magnitude of $\beta_k$ is proportional to the condition number of the basis $\hat X_k$ after column scaling. From~\cref{sec:dang_eigvals,sec:twice_is_enough}, we know that $\beta_1\approx u/d$ and $\beta_k\approx u$ for $k\geq 2$, provided that $\mathcal{\hat S}_k$ and $\mathcal{V}$ do not become too close to orthogonal during the iteration.

\subsection{Stability and stagnation}\label{sec:no_stagnation}

According to~\cref{thm:1step_perturbed}, the search subspace is refined by a factor comparable to~\cref{thm:1step_standard}, up to the size of the errors $\beta_k$ introduced in $\mathcal{V}^\perp$, at each iteration. As we accumulate iterations, the errors in $\mathcal{V}^\perp$ are filtered out by $r(A)$ and, in the apt words of the authors of~\cite{peter2014feast}, ``the dominant error term is the one most recently introduced." As long as $\cos\theta_1(\mathcal{\hat S}_k,\mathcal{V})$ is bounded sufficiently far from zero for $k\geq 0$, the sequences $\alpha_k$ and $\beta_k$ remain stable at the order of $u/d$ and $u$ (respectively) after the first iteration. In this case, we expect that $\mathcal{\hat S}_k$ converges geometrically toward $\mathcal{V}$ until a threshold of about $u$ is reached, after which convergence stagnates. This is what we observe in~\cref{fig:arnoldi_fsi,fig:contour_refine}.

If $\cos\theta_1(\mathcal{\hat S}_k,\mathcal{V})$ does become very small at some step in the iteration, then the one-step refinement bound may not imply any refinement in the search subspace at all: the iteration in~\cref{eqn:ratSI_perturbed} is potentially unstable. With a slight change of perspective, we now characterize the behavior of the iterates in~\cref{eqn:ratSI_perturbed} as $k\rightarrow\infty$, addressing both the stability and the threshold for stagnation in subspace refinement.

Let us introduce the constants $\rho=|r(\lambda_{m+1})|/|r(\lambda_m)|$, $\epsilon_1=\gamma_1\|A\|u/d$, and $\epsilon_2=\gamma_2\|A\|\hat Mu$, where $\hat M$ is an $\mathcal{O}(1)$ uniform bound on $\kappa(X_kC_k)$ for $k\geq 2$ (i.e., from ~\cref{thm:twice-is-enough}). Consider the function
\begin{equation}\label{eqn:discrete_map}
\Phi(\eta)=\frac{1}{1-\epsilon_2}\left(\frac{\rho\,\eta}{1-\epsilon_1(1+\eta)}+\epsilon_2\right).
\end{equation}
Because $1/\cos\theta\leq 1+\tan\theta$ when $0\leq\theta\leq\pi/2$, we can rewrite~\cref{thm:1step_perturbed} in the form $\tan\theta_1(\mathcal{\hat S}_k,\mathcal{V})\leq \Phi(\tan\theta_1(\mathcal{\hat S}_{k-1},\mathcal{V}))$ when $k\geq 2$. We can understand the ``worst-case" behavior of subspace iteration by studying the trajectory of $\tan\theta_1(\mathcal{\hat S}_0,\mathcal{V})$, for some initial subspace $\mathcal{\hat S}_0$, obtained by iterating the map $\Phi$.

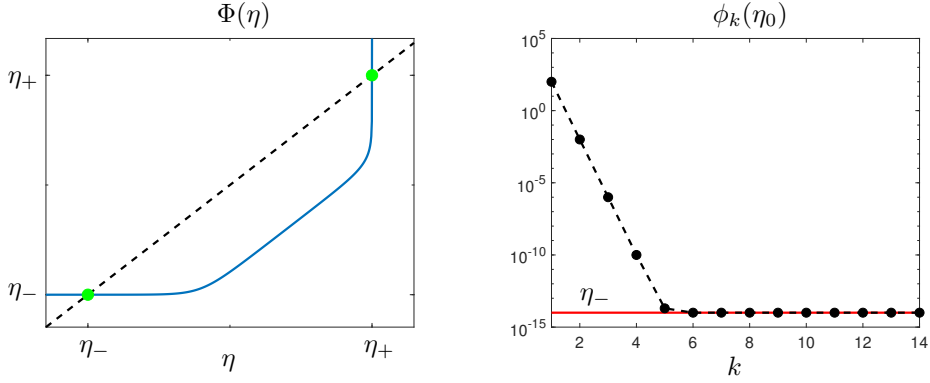
\begin{figure}[!tbp]
  \centering
  \begin{minipage}[b]{0.48\textwidth}
    \begin{overpic}[width=\textwidth]{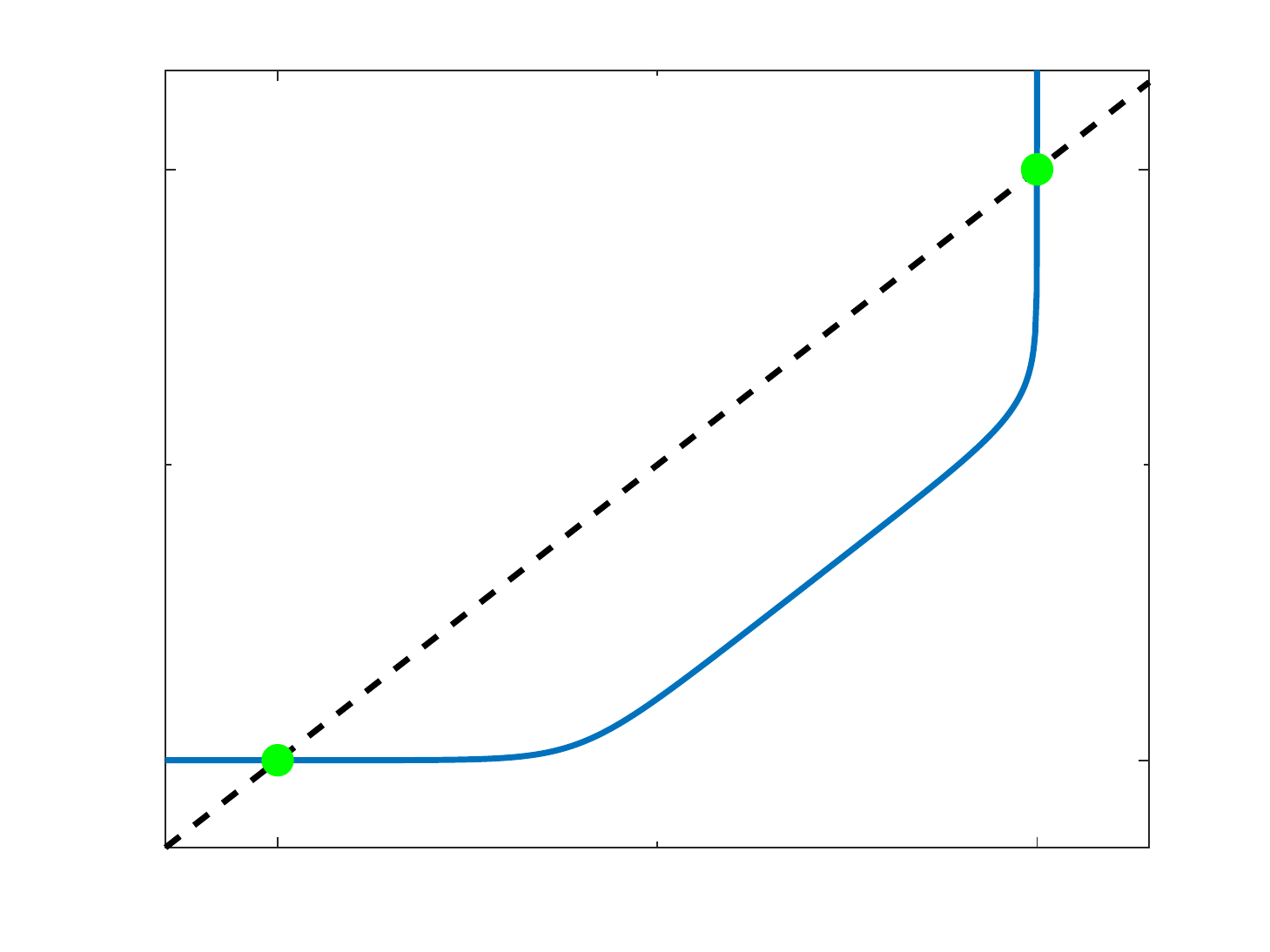}
     \put (49,73) {$\displaystyle \Phi(\eta)$}
     \put (50,0) {$\displaystyle \eta$}
     \put (20,3) {$\displaystyle \eta_-$}
     \put (80,3) {$\displaystyle \eta_+$}
     \put (5,60) {$\displaystyle \eta_+$}
     \put (5,15) {$\displaystyle \eta_-$}
     \end{overpic}
  \end{minipage}
    \hfill
  \begin{minipage}[b]{0.48\textwidth}
    \begin{overpic}[width=\textwidth]{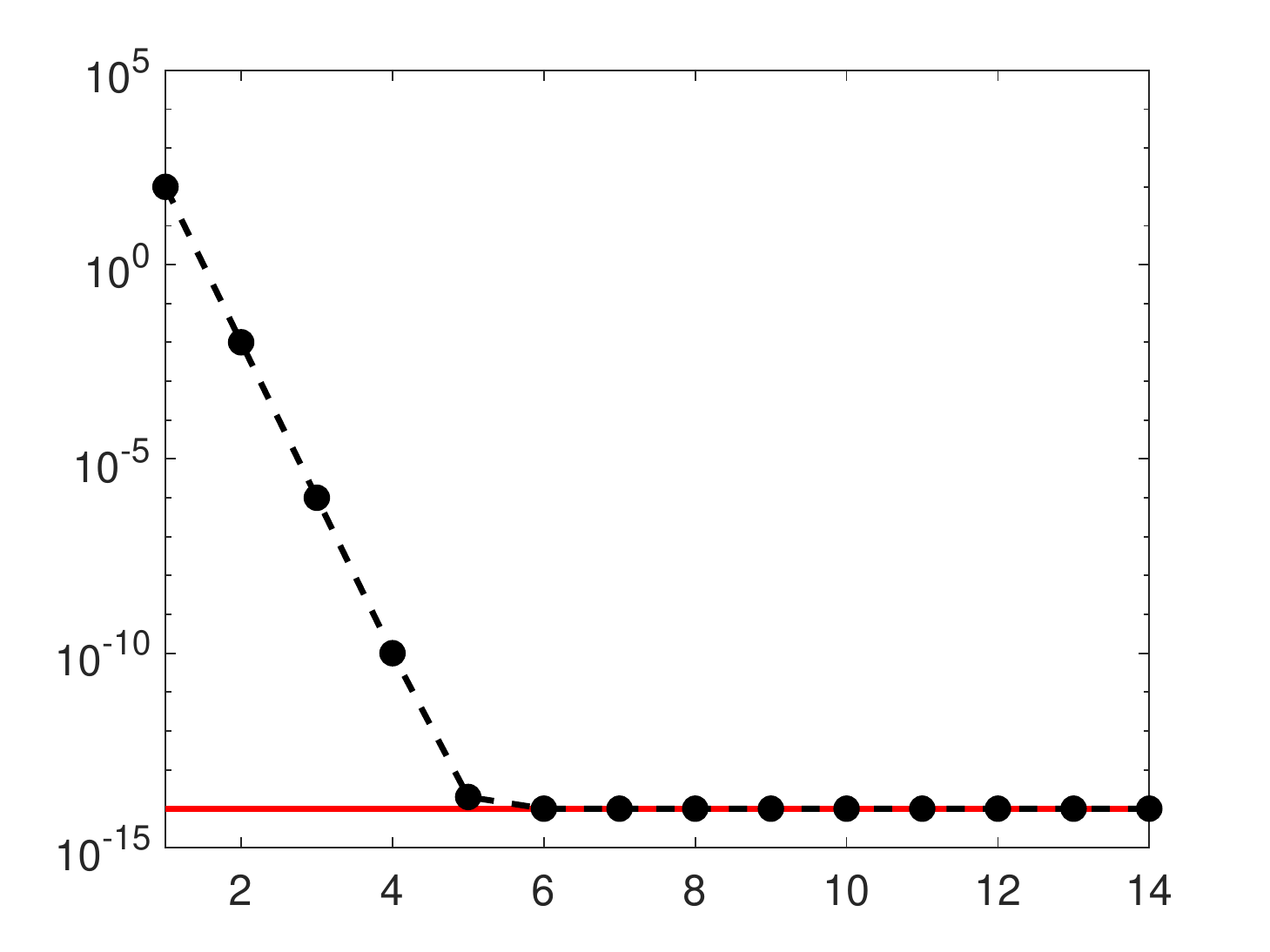}
   	 \put (47,73) {$\displaystyle \phi_k(\eta_0)$}
     \put (50,-2) {$\displaystyle k$}
     \put (19,14) {$\displaystyle \eta_-$}
     \end{overpic}
  \end{minipage}
  \caption{
The dynamics of perturbed subspace iteration from~\cref{eqn:discrete_dynamics}. In the left panel, the solid line is the graph of $\Phi(\eta)$ and its fixed points (green circles) are marked at the intersections $\Phi(\eta_\pm)=\eta_\pm$. If $\tan\theta_1(\mathcal{\hat S}_0,\mathcal{V})$ falls between the two fixed points (green circles), then the $\tan\theta_1(\mathcal{\hat S}_k,\mathcal{V})$ must converge geometrically to a threshold near the lower fixed point (see~\cref{thm:perturbed_convergence}). In the right panel, the iterated map $\phi_k(\eta_0)$ (circles) is compared with the upper bound in~\cref{thm:perturbed_convergence} (dashed line) for $k=0,\ldots,14$. For this experiment, $\eta_0=100$, $\epsilon_1=10^{-5}$, $\epsilon_2=10^{-14}$, and $\rho=10^{-4}$.
\label{fig:stability}}
\end{figure}

Given $\eta_0>0$, let $\phi_k(\eta_0)$ denote the $k$-fold iteration of the map $\Phi$ on the point $\eta_0$, so that (letting $f\circ g$ denote the composition of two functions) we have
\begin{equation}\label{eqn:discrete_dynamics}
\phi_k(\eta_0)=\underbrace{\Phi\circ\cdots\circ\Phi}_k(\eta_0).
\end{equation}
We call $\eta_*$ a fixed point of $\Phi$ if $\Phi(\eta_*)=\eta_*$ and say that $\eta_*$ is monotone attracting for $\Omega\subset[0,\infty)$ if $\phi_k(\eta)\rightarrow\eta_*$ monotonically as $k\rightarrow\infty$ for all $\eta\in\Omega$. After applying $\Phi$ to $\tan\theta_1(\mathcal{\hat S}_0,\mathcal{V})$ $k$ times, we see that (since $\Phi(\eta)$ is non-decreasing on $0\leq\eta<-1+1/\epsilon_1$)
\begin{equation}\label{eqn:dynamics-convergence_link}
\tan\theta_1(\mathcal{\hat S}_k,\mathcal{V})\leq \phi_k\left(\tan\theta_1(\mathcal{\hat S}_0,\mathcal{V})\right),
\end{equation}
as long as $\phi_j(\tan\theta_1(\mathcal{\hat S}_0,\mathcal{V}))<-1+1/\epsilon$ for each $j\geq 1$. Consequently, the fixed points of $\Phi$ and their attracting sets provide insight into the behavior of $\tan\theta_1(\mathcal{\hat S}_k,\mathcal{V})$ in the limit $k\rightarrow\infty$, that is, about the convergence and stability of the iteration in~\cref{eqn:ratSI_perturbed}.

\begin{lemma}\label{lem:fixed_points}
Define the map $\Phi:[0,-1+1/\epsilon_1)\rightarrow[0,\infty)$ as in~\cref{eqn:discrete_map}, with constants $\rho>0$ and $\epsilon_1$, $\epsilon_2\geq 0$. Let
$$
\delta=\left(1-\rho(1-\epsilon_2)^{-1}-\epsilon_1(1-\epsilon_2)\right)/(2\epsilon_1), \qquad\text{and}\qquad \sigma=\epsilon_2(1-\epsilon_1)/\epsilon_1.
$$
If $\delta^2>\sigma$, then $\Phi$ has precisely two fixed points, given by $\eta_\pm = \delta\pm\sqrt{\delta^2-\sigma}$. Moreover, the fixed point $\eta_-$ is monotone attracting on $[0,\eta_+)$.
\end{lemma}
\begin{proof}
Starting from the fixed point equation $\Phi(\eta_*)=\eta_*$, we multiply through by $1-\epsilon_1(1+\eta_*)$ to obtain the quadratic equation
\begin{equation}\label{eqn:quadratic_eqn}
\epsilon_1\eta_*^2 - [1-\rho(1-\epsilon_2)^{-1} - \epsilon_1(1-\epsilon_2)]\eta_* + \epsilon_2(1-\epsilon_1)=0.
\end{equation}
Applying the quadratic formula for the roots and rewriting in terms of $\delta$ and $\sigma$ concludes the fixed-point calculation. Now, the quadratic on the left-hand side of~\cref{eqn:quadratic_eqn} is negative between the roots, which implies that $\Phi(\eta)>0$ for $0<\eta<\eta_-$ and $\Phi(\eta)<0$ for $\eta_-<\eta<\eta_+$. The change of sign at each fixed point implies that $\eta_-$ attracts nearby points and that $\eta_+$ repels nearby points. Because $\Phi$ is non-decreasing and has no other fixed points, we conclude that $\eta_-$ is monotone attracting on $[0,\eta_+)$.
\end{proof}

\Cref{lem:fixed_points} shows that if $\tan\theta_1(\mathcal{\hat S}_0,\mathcal{V})< \eta_+$, then $\tan\theta_1(\mathcal{\hat S}_k,\mathcal{V})$ must eventually be on the order of $\eta_-$ or better for all sufficiently large $k$. Recalling that the constants $\epsilon_1$ and $\epsilon_2$ are on the order of $u/d$ and $u$, respectively, and that $\rho$ is the filtered spectral ratio, we estimate the size of the fixed points to be
\begin{equation}
\eta_- \approx \frac{\epsilon_2}{1-\rho}, \qquad\text{and}\qquad \eta_+ \approx -1 +\frac{1-\rho}{\epsilon_1}.
\end{equation}
Crucially, the lower fixed point $\eta_-$ is on the order of $u$, not $u/d$. Having established stability properties of the perturbed iteration in~\cref{eqn:ratSI_perturbed}, we can now estimate the rate of convergence to the fixed point $\eta_-$.

\begin{theorem}\label{thm:perturbed_convergence}
Define the map $\Phi:[0,-1+1/\epsilon_1)\rightarrow[0,\infty)$ as in~\cref{eqn:discrete_map}, with constants $\rho>0$ and $\epsilon_1$, $\epsilon_2\geq 0$. Let $\phi_k$ denote the $k$-fold iteration of $\Phi$ as in~\cref{eqn:discrete_dynamics}. If $\Phi$ satisfies the hypotheses of~\cref{lem:fixed_points}, then given $0\leq\eta_0<\eta_+$, it holds that
\begin{equation}\label{eqn:perturbed_convergence}
\phi_k(\eta_0)\leq\tilde\rho^k\eta_0 + \tilde\epsilon_2(1-\tilde\rho)^{-1}, \qquad k\geq 1.
\end{equation}
Here, $\tilde\rho=\rho(1-\epsilon_2)^{-1}(1-\epsilon_1(1+\eta_0))^{-1}$ and $\tilde\epsilon_2=\epsilon_2(1-\epsilon_2)^{-1}$.
\end{theorem}
\begin{proof}
Denote $\eta_k=\phi_k(\eta_0)$, for each $k\geq 1$. From the definitions of $\Phi$ and $\phi_k$ in~\cref{eqn:discrete_map,eqn:discrete_dynamics}, respectively, we compute that
\begin{equation}\label{eqn:uniform_bound}
\eta_k=\Phi(\eta_{k-1}) = \tilde\rho\eta_{k-1}+\tilde\epsilon_2\, \qquad k\geq 1.
\end{equation}
By hypothesis,~\cref{lem:fixed_points} applies, so $\eta_k\rightarrow\eta_-$ monotonically as $k\rightarrow\infty$ and, consequently, $\tilde\rho<1$. Therefore, we iterate~\cref{eqn:uniform_bound} $k-1$ times to obtain
$$
\eta_k = \tilde\rho^{\,k}\eta_0+\tilde\epsilon_2\sum_{j=0}^{k-1}\tilde\rho^{\,j} \leq \tilde\rho^{\,k}\eta_0+\tilde\epsilon_2/(1-\tilde\rho).
$$
Plugging the original parameters back in to $\tilde\rho$ and $\tilde\epsilon_2$ establishes~\cref{eqn:perturbed_convergence}.
\end{proof}

Thus,~\cref{thm:perturbed_convergence} and~\cref{eqn:dynamics-convergence_link} demonstrate that the reduction of $\tan\theta_1(\mathcal{\hat S}_k,\mathcal{V})$ down to the order of $\eta_-$ is approximately geometric with rate close to $\rho$. So (accounting for the fact that the additive perturbation term is actually on the order of $u/d$ in the first iteration) it takes approximately $1+\log(\eta_-)/\log(\rho)$ steps for $\mathcal{\hat S}_k$ to converge to within order $u$ of $\mathcal{V}$, as measured by the tangent of the principal angle between the two subspaces.

\section{Non-normal matrices}\label{sec:non-normal_case}

We now consider the case of an $n\times n$ diagonalizable matrix $A$ whose eigenvectors are not orthogonal. Although a straightforward extension of~\cref{lem:X1_condition} shows that the condition number of $X_1$ still scales, generically, like $1/d$ (see~\cref{lem:X_1_condition-nn} below), the effect of a dangerous eigenvalue on subsequent iterates, $X_2,X_3,\ldots$, computed via~\cref{eqn:ratSI} is distinct in the non-normal case due to interactions among non-orthogonal modes. In fact, the condition numbers of the computed iterates do not improve during subsequent iterations unless approximate eigenvectors (i.e., from Ritz vectors) are incorporated into the subspace iteration (see~\cref{alg:FSI_w_RR}). Even with this modification, the condition numbers may remain large after one iteration when $d$ is very small (loosely, when $d\ll\sqrt{u}$), unlike the normal case. Here, we demonstrate that $\kappa(X_k)$ is typically reduced in step with the error in the Ritz vectors and that $\kappa(X_k)\approx(u/d)^k$ in the best case (i.e., when $|r(\lambda_{m+1})|/|r(\lambda_m)|\approx u$ and the Ritz vectors are well-conditioned at each iteration).

When $A$ does not have an orthogonal basis of eigenvectors (but is still diagonalizable), the orthogonal spectral projectors $v_iv_i^*$ that diagonalize the filter in~\cref{eqn:amplify_v1} are replaced by oblique spectral projectors, so that
\begin{equation}\label{eqn:v1w1_amplify}
r(A)x=\sum_{i=1}^n r(\lambda_i)\frac{w_i^*x}{w_i^*v_i}v_i=\frac{w_1^*x}{(de^{i\theta})(w_1^*v_1)}v_1+\mathcal{O}(1),\qquad\text{as}\qquad d\rightarrow 0.
\end{equation}
Here, $w_1,\ldots,w_n$ are the left eigenvectors of $A$, satisfying $w_i^*A=\lambda_iw_i^*$ with $\|w_i\|=1$ for $i=1,\ldots,n$. Likewise, the spectral decomposition in~\cref{eqn:eigendecomposition} is replaced by 
\begin{equation}\label{eqn:eigendecomposition2}
A=V_1\Lambda_1W_1^* + V_1\Lambda_2W_2^*,
\end{equation}
where the $i$th column of $W=[W_1\,\,W_2]$ is $(w_i^*v_i)^{-1}w_i$. With this normalization, $V$ and $W$ form a biorthogonal system, meaning that $W^*V=I$, $I$ being the $n\times n$ identity matrix. In the biorthogonal system, the dangerous eigenvalue amplifies the $w_1$ component in the input $x$ along the $v_1$ direction in the output $r(A)x$. Due to biorthogonality, $v_1$ and $w_1$ are parallel only when $v_1$ is orthogonal to $v_2,\ldots,v_n$.

\subsection{First iteration}\label{sec:nn-first_iteration}

To develop a sense of how non-normality impacts the conditioning of the iterates, it is worthwhile to revisit the analysis of $\kappa(X_1)$ in~\cref{lem:X1_condition} when $A$ is only diagonalizable. While the condition number of $X_1$ is still $\mathcal{O}(1/d)$ as $d\rightarrow 0$, the constants in the bound now depend on the structure of the left and right eigenvectors. This is because the stretching and shrinking actions of $A$ no longer belong solely to its eigenvalues, but can be enhanced or attenuated by interactions among non-orthogonal eigenvectors. We denote the smallest singular values of $V_1$ and $W_1$ by $\sigma_m(V_1)$ and $\sigma_m(W_1)$, respectively.

\begin{proposition}
\label{lem:X_1_condition-nn}
Let diagonalizable $A\in\mathbb{C}^{n\times n}$ and $r:\Lambda\rightarrow\mathbb{C}$ satisfy~\cref{eqn:eigendecomposition2} and~\cref{eqn:eig_index}, respectively, and given orthonormal $Q_0\in\mathbb{C}^{n\times m}$, let $X_1=r(A)Q_0$. If $U_1={\rm qf}(W_1)$ and $U_1^*Q_0$ has full rank, then the condition number of $X$ satisfies
\begin{equation}\label{eqn:X_1_condition-nn}
\frac{\|w_1^*Q_0\|/|w_1^*v_1|}{d\kappa(V)|r(\lambda_2)|}\lesssim \kappa(X_1)\leq\Big|\frac{r(\lambda_1)}{r(\lambda_m)}\Big|\frac{\kappa(V)\|(U_1^*Q_0)^{-1}\|}{\sigma_m(V_1)\sigma_m(W_1)}, \quad\text{as}\quad d\rightarrow 0.
\end{equation}
\end{proposition}
\begin{proof}
The steps of the proof are essentially identical to those in~\cref{lem:X1_condition} if~\cref{eqn:v1w1_amplify,eqn:eigendecomposition2} are used in place of~\cref{eqn:amplify_v1,eqn:eigendecomposition}, so we emphasize the adaptations made for non-orthogonal eigenvectors. For the largest singular value of $X_1$, we bound $\sigma_1(X_1)=\|r(A)Q_0\|\leq\kappa(V)|r(\lambda_1)|$, since $|r(\lambda_1)|\leq\|r(A)\|\leq\kappa(V)\|r(\Lambda)\|$ in the non-normal case. If we use~\cref{eqn:eigendecomposition2} to decompose $X_1$ as in~\cref{eqn:spec_decomp2}, the singular values of $r(A)W_1^*Q_0$ do not tell us directly about the singular values of $X_1$ because $V$ is not unitary. However, if $\Omega_1 R_1=V_1$ and $\Omega_2R_2=V_2$ are economy-sized QR factorizations, we can decompose
$$
r(A)Q_0=\begin{bmatrix} \Omega_1 & \Omega_2 \end{bmatrix}
\begin{bmatrix}
R_1r(\Lambda_1)W_1^*Q_0 \\ R_2r(\Lambda_2)W_2^*Q_0 
\end{bmatrix}.
$$
Since $\Omega_1$ and $\Omega_2$ have orthonormal columns, we apply the argument in the proof of~\cref{lem:X1_condition} to obtain the bound $1/\sigma_m(X_1)\leq\|(R_1r(\Lambda_1)W_1^*Q_0)^{-1}\|$. Now, $R_1$ has the same singular values as $V_1$ and $\|R_1^{-1}\|=1/\sigma_m(R_1)$, so we have that
\begin{equation}\label{eqn:intermediate_nn_bound}
\kappa(X_1)\leq\frac{|r(\lambda_1)|}{|r(\lambda_m)|}\frac{\kappa(V)\|(W_1^*Q_0)^{-1}\|}{\sigma_m(V_1)}.
\end{equation}
The upper bound in~\eqref{eqn:X_1_condition-nn} follows by substituting the QR decomposition $U_1S_1=W_1$ into~\cref{eqn:intermediate_nn_bound} and noting that $\|S_1^{-1}\|=1/\sigma_m(W_1)$.

A lower bound on $\sigma_1(X_1)$ follows directly from~\cref{eqn:v1w1_amplify}, analogous to~\cref{eqn:X1_asymptotics}. For the lower bound on $1/\sigma_m(X_1)$, we can use~\cref{eqn:v1w1_amplify} to write $X_1$ as a rank one perturbation of the matrix
$$
\tilde N_2 = V{\rm diag}(0,\lambda_2,\ldots,\lambda_n)W^*Q_0.
$$
We have that $\sigma_1(N_2)\leq\|V\|\|W^*\||r(\lambda_2)|=\kappa(V)|r(\lambda_2)|$, where the equality is due to biorthogonality, which implies that $W^*=V^{-1}$. By interlacing, we find that $1/\sigma_m(X_1)\geq 1/(\kappa(V)|r(\lambda_2)|)$, establishing the asymptotic lower bound in~\cref{eqn:X_1_condition-nn}.
\end{proof}

When $A$ is normal,~\cref{lem:X_1_condition-nn} reduces to~\cref{lem:X1_condition}. In the non-normal case, ill-conditioning in the eigenvectors, reflected in $\kappa(V)$, widens the interval between the upper and lower bounds. Similarly, ill-conditioning in the target eigenvectors, captured by the smallest singular values of $V_1$ and $W_1$ (since the columns of both matrices have unit norm), may further widen the gap. On the other hand, the dangerous eigenvalue itself is ill-conditioned when $|w_1^*v_1|$ is small.\footnote{With $\|v_i\|=\|w_i\|=1$, the quantity $|w_i^*v_i|^{-1}$ is Wilkinson's condition number for $\lambda_i$, measuring the first-order sensitivity of the eigenvalue to infinitesimal perturbations in $A$~\cite[pp.~88--89]{wilkinson1965algebraic}.} The left-hand side of~\cref{eqn:X_1_condition-nn} illustrates how this may enhance the amplifying effects of the dangerous eigenvalue, increasing the asymptotic lower bound to $d|w_1^*v_1|^{-1}$. Broadly speaking, the widening gap between upper and lower bounds indicates that our picture is blurred in the non-normal case because the structure of the eigenvectors plays a key role. The extent of the damage may depend on where the ill-conditioning in $V$ is concentrated.

\subsection{Iterating with orthonormal bases}\label{sec:nn-ONB}

\begin{figure}[!tbp]
  \centering
  \begin{minipage}[b]{0.48\textwidth}
    \begin{overpic}[width=\textwidth]{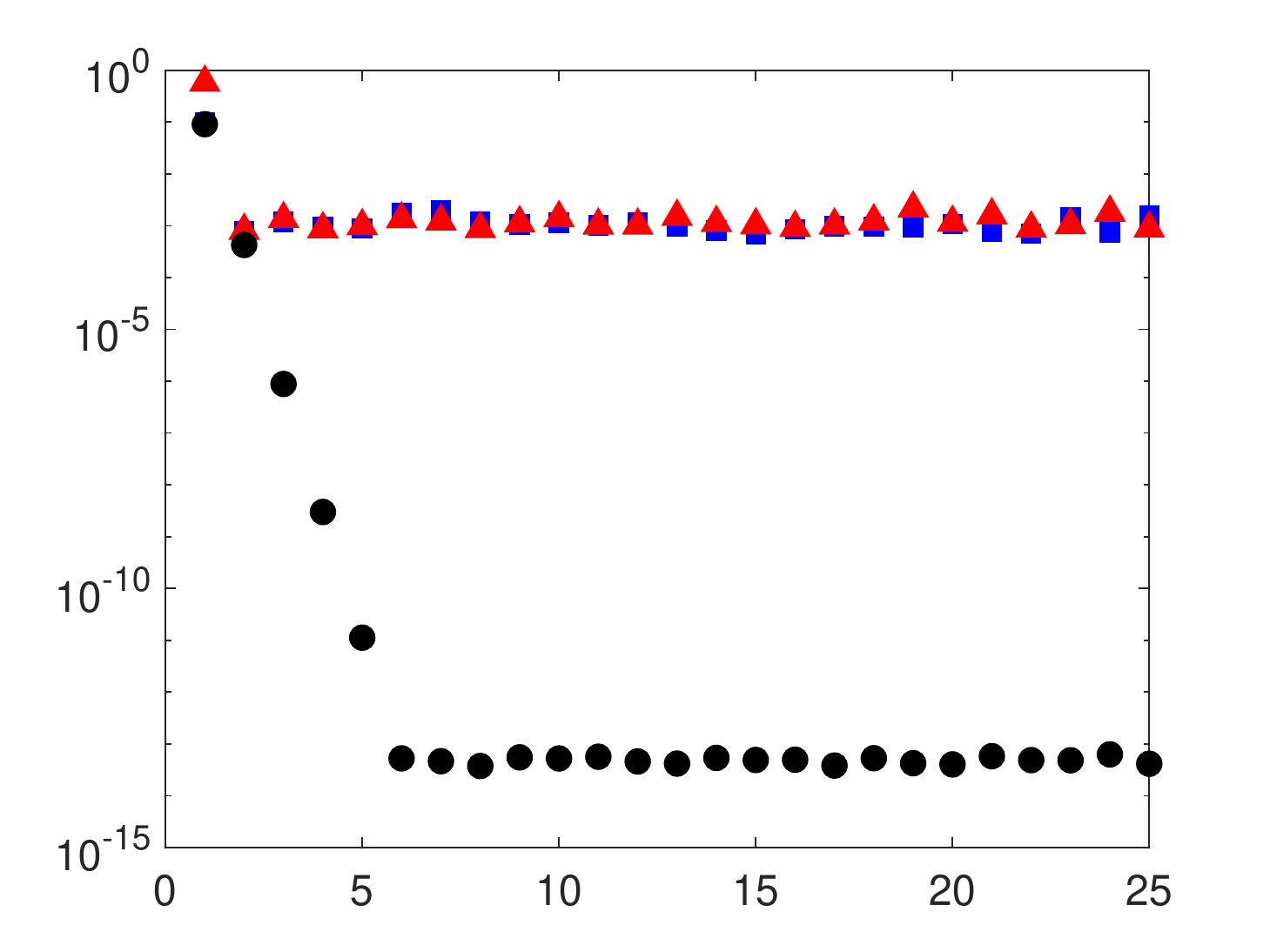}
   	 \put (30,74) {$\displaystyle \max_i\|A\hat v_i-\hat \lambda_i\hat v_i\|$}
     \put (50,-2) {$\displaystyle k$}
     \end{overpic}
  \end{minipage}
  \hfill
  \begin{minipage}[b]{0.48\textwidth}
    \begin{overpic}[width=\textwidth]{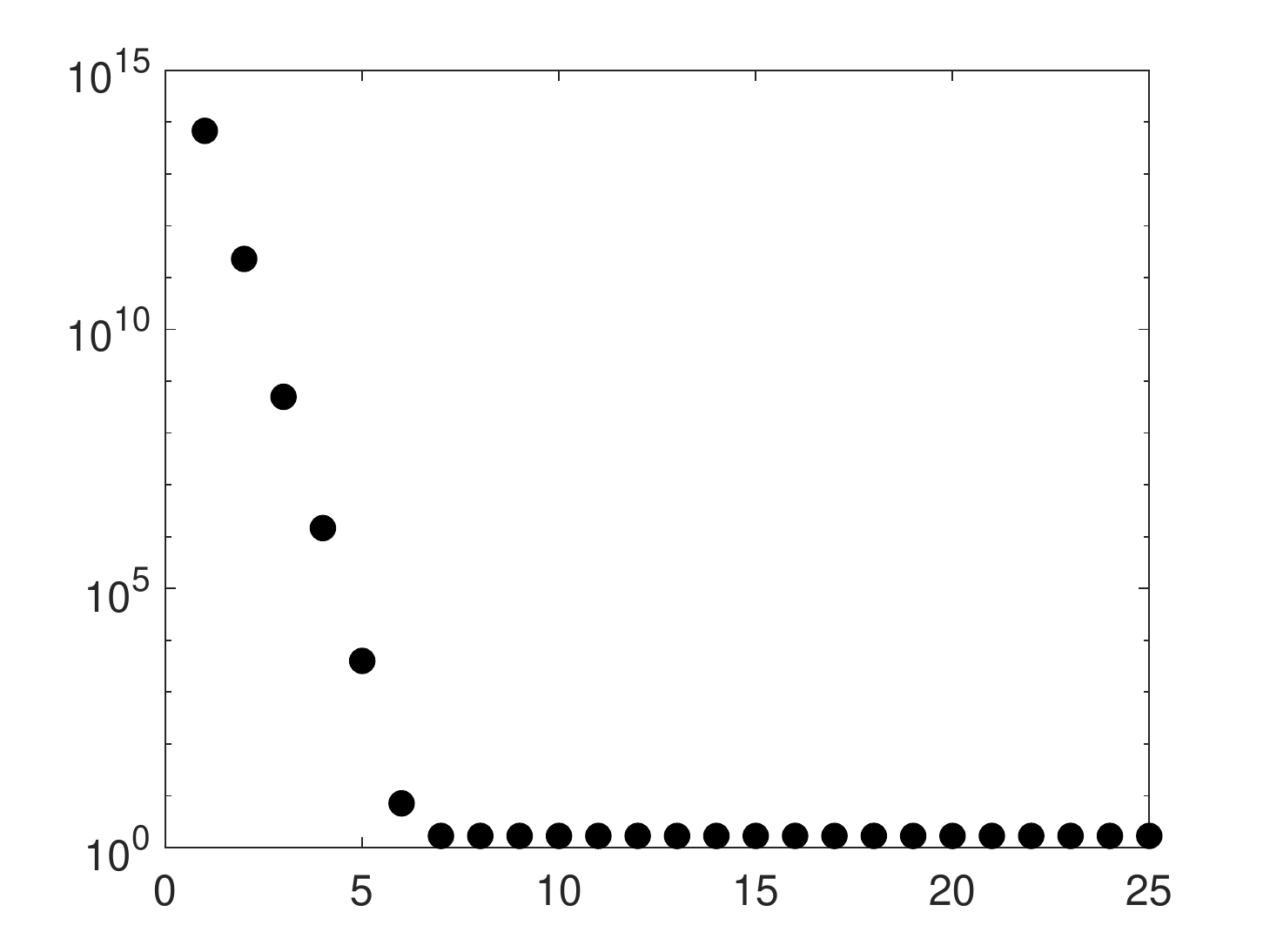}
    \put (41,73) {$\displaystyle \kappa(\hat Z_kD_k)$}
     \put (50,-2) {$\displaystyle k$}
     \end{overpic}
  \end{minipage}
  \caption{
\textit{Dangerous eigenvalues of a non-normal matrix.} The eigenvalues and rational filter are identical to the setup displayed in~\cref{fig:exp1_setup}, however, this matrix has non-orthogonal eigenvectors and the dangerous eigenvalue has been moved to distance $d=10^{-13}$ from the pole at $z=10$. On the left, the maximum residual of $10$ target eigenpairs after each iteration of~\cref{eqn:ratSI} (blue squares), a variant of subspace iteration based on Schur vectors~\cite[ch.~5.2]{saad2011numerical} (red triangles), and a variant based on approximate eigenvectors, described in~\cref{alg:FSI_w_RR} (black circles). On the right, the condition number of the iterates $\hat Z_kD_k$ ($D_k$ scales the columns of $\hat Z_k$ to have unit norm) decreases in step with residuals from~\cref{alg:FSI_w_RR}, at a rate of about $u/d$ per iteration.
\label{fig:exp3_results}}
\end{figure}

Now that we understand the interaction between non-normality and dangerous eigenvalues in the initial iteration, we are ready to examine subsequent iterations. As in~\cref{sec:twice_is_enough}, we focus on the coordinates of $Q_1$ in the eigenvector basis, partitioned into blocks as
\begin{equation}\label{eqn:Q1_structure_nn}
W_1^*Q_1 = \begin{bmatrix}
w_1^*q_1^{(1)} & w_1^*\tilde Q_1 \\
\tilde W_1^*q_1^{(1)} & \tilde W_1^*\tilde Q_1
\end{bmatrix}.
\end{equation}
The critical observation about~\cref{eqn:Q1_structure_nn} is that, in contrast to the normal case, the upper right block is not small (the lower-left block remains small). Although the columns of $\tilde Q_1$ are still nearly orthogonal to $v_1$, the eigenvectors $v_1$ and $w_1$ are only parallel in the special case that $v_1$ is orthogonal to $v_2,\ldots,v_n$. Consequently, $w_1^*\tilde Q_1$ is typically $\mathcal{O}(1)$ and, when we compute $X_2=r(A)Q_1$, the components in each column of $Q_1$ in the $w_1$ direction will be amplified according to~\cref{eqn:v1w1_amplify}. Each column of $X_2$ will be dominated by $v_1$ at magnitude $\mathcal{O}(1/d)$ and $X_2$ is just as ill-conditioned as $X_1$ in the first iteration. This line of thinking seems to indicate that, when $r(A)$ is repeatedly applied to an orthonormal basis, subspace iteration for non-normal matrices must stagnate at an accuracy of $\approx u/d$ due to ill-conditioning in the iterates $X_1,X_2,\ldots$.

To illustrate, we return to the experimental setup illustrated in~\eqref{fig:exp1_setup}. We select the same rational filter and a matrix with the same eigenvalues, but now the eigenvector matrix is not orthogonal. The condition number of the eigenvector matrix is $\approx 10^{2}$, but the target eigenvectors themselves are not far from orthogonal. \Cref{fig:exp3_results} shows the maximum residual of the computed target eigenpairs after each of the first $10$ iterations of~\cref{eqn:ratSI}. We also compare with a modified subspace iteration based on Schur vectors that is commonly used to compute eigenvalues of non-normal matrices~\cite[ch.~5.2]{saad2011numerical}. Both iterations apply the rational filter directly to an orthonormal basis for the search space and the residuals stagnate near $u/d$ in both cases.

\subsection{Iterating with approximate eigenvectors}\label{sec:nn-ritz_vec_rescue}

What can we do to improve the conditioning of the iterates and the accuracy in the target eigenpairs? Consider another common variant of subspace iteration shown in~\cref{alg:FSI_w_RR}, which forms the iterates $Z_1,Z_2,\ldots$ by applying $r(A)$ to approximate eigenvectors constructed from the Ritz vectors at each iteration. Let us partition $W_1^*Y_1$ in the usual way,
\begin{equation}\label{eqn:Y1_structure_nn}
W_1^*Y_1 = \begin{bmatrix}
w_1^*y_1^{(1)} & w_1^*\tilde Y_1 \\
\tilde W_1^*y_1^{(1)} & \tilde W_1^*\tilde Y_1
\end{bmatrix}
=\begin{bmatrix}
e & f \\
g & H
\end{bmatrix},
\end{equation}
where $\tilde W_1$ and $\tilde Y_1$ denote the last $m-1$ columns of $W_1$ and $Y_1$, respectively. Now, because the left and right eigenvectors are biorthogonal, $w_1^*$ annihilates the remaining target eigenvectors $v_2,\ldots,v_m$, so the upper right block $f$ in~\cref{eqn:Y1_structure_nn} is small when the columns of $Y_1$ are a good approximation to the target eigenvectors. In turn, small $\|f\|$ mitigates the amplification of $v_1$ in the last $m-1$ columns of $Z_2$. 

\begin{algorithm}[t]
\textbf{Input:} Given $A\in\mathbb{C}^{n\times n}$, $r:\Lambda\rightarrow\mathbb{C}$, and $Y_0\in\mathbb{C}^{n\times m}$.  \\
\vspace{-4mm}
\begin{algorithmic}[1]
\FOR{$k=1,2,\ldots$}
	\STATE Apply the filter $Z_k=r(A)Y_{k-1}$.
	\STATE Compute orthonormal basis $Q_k={\rm qf}(Z_k)$.
	\STATE Form $A_k=Q_k^*AQ_k$ and diagonalize $A_k=U_k\Theta_kU_k^{-1}$.
	\STATE Set $Y_k=Q_kU_k$.
\ENDFOR
\end{algorithmic} \textbf{Output:} Approximate eigenvalue matrix $\Theta_k$ and eigenvector matrix $Y_k$.
\caption{Filtered subspace iteration with Rayleigh--Ritz projection.}\label{alg:FSI_w_RR}
\end{algorithm}

Unfortunately, the behavior of approximate eigenvectors computed with~\eqref{alg:FSI_w_RR} may vary widely for general non-normal matrices. In exact arithmetic, their accuracy will depend on the rational filter through the eigenvalues of $r(A)$ and on interactions among non-orthogonal eigenvectors. This can delay convergence and may lead to instability on a computer. In floating-point arithmetic, it is further limited by the accuracy in the computed orthonormal basis and Ritz vectors. Despite these difficulties, we can glean some practical insight into a distinct feature of the non-normal setting by examining a ``best-case" situation.

Let us suppose that the non-normal effects are relatively mild, that $r(\cdot)$ filters out the unwanted eigenvalues to unit round-off or better (as in~\cref{fig:exp1_setup}), and that the Ritz vectors are computed accurately at each iteration. In this regime, the accuracy of the approximate eigenvectors $Y_1$ is limited mainly by the accuracy in the computed orthonormal basis, $\hat Q_1$, and we can focus on the influence of the dangerous eigenvalue in the second iteration (and beyond). From our analysis of the first iteration in~\cref{sec:nn-first_iteration}, we expect that $\|\hat Q_1-Q_1\|\approx u/d$ and, therefore, (by our assumptions on the filter and the Ritz vectors) that $\|\hat Y_1 - V_1\|\approx u/d$.

Interestingly, the order of magnitude of block $f$ in~\eqref{eqn:Y1_structure_nn} is distinctly different from the analogous block $b$ in the normal case. Instead of the perfect balancing between $b$ and $r(\lambda_1)$ when the filter is applied (leading to perfectly well-conditioned columns of $X_2$), we have the order-of-magnitude estimate $\|f\||r(\lambda_1)|\approx u/d^2$. In other words, $v_1$ may still dominate each column of $Z_2$ when $d\ll\sqrt{u}$, but the gap in magnitude between the $v_1$ component and the remaining target components in the last $m-1$ columns is reduced by a factor of $u/d$ at the second iteration.~\Cref{fig:exp3_results} illustrates this phenomenon in action with the same matrix and rational filter used for the experiments in~\cref{sec:nn-ONB}. The residuals in the target eigenpairs decrease geometrically with rate $u/d$ (left panel), mirroring the reduction in the condition number of the iterates $Z_k$ (after scaling columns to have unit norm, right panel).

Thus, for a mildly non-normal matrix with a dangerous eigenvalue at distance $d\ll\sqrt{u}$ from a pole of $r(\cdot)$, two iterations are not usually enough to remove the adverse influence of the dangerous eigenvalue. Instead, the target residuals and the errors in the computed orthonormal basis are often refined in step down to the unit round-off (depending on the sensitivity of the target eigenpairs). As in the normal case, round-off errors caused by the dangerous eigenvalue may even go unnoticed when the rational filter is mediocre so that the noise in the unwanted directions is dominated by poor filtering.

\section{Restarting Arnoldi}\label{sec:rat_Krylov}

Now that we understand the right-hand side of~\cref{fig:arnoldi_fsi}, let us examine the stagnation of Arnoldi with shift-and-invert enhancement illustrated in the left-hand panel of the same figure. Unlike subspace iteration, which applies $r(A)$ iteratively to a subspace of fixed dimension, Arnoldi refines the subspace by expanding it.
Given an initial unit vector $q_1\in\mathbb{C}^n$, shift-and-invert Arnoldi computes the iterates
\begin{equation}\label{eqn:arnoldi}
y_k=s(A)q_{k-1}, \qquad q_k={\rm mgsr}(y_k\,;\,q_1,\ldots,q_{k-1}),
\end{equation}
with the expression ${\rm mgrs}(\cdot)$ indicating that $y_k$ is orthogonalized against $q_1,\ldots,q_{k-1}$ using modified Gram--Schmidt with full reorthogonalization~\cite[pp.~307--308]{stewart2001matrix}.

After $k$ steps of~\cref{eqn:arnoldi}, we have an $n\times k$ orthonormal basis $Q_k=[q_1\,\cdots\,q_k]$ and we can approximate eigenpairs of $A$ in one of two ways:
\begin{itemize}[leftmargin=*,noitemsep]
\item Directly from the eigenpairs of the upper Hessenberg matrix $H_k$ generated from the weights calculated during modified Gram--Schmidt~\cite[p.253]{trefethen1997numerical}. 
\item A Rayleigh--Ritz step by computing eigenpairs of $A_k=Q_k^*AQ_k$. 
\end{itemize}
Usually, the upper Hessenberg matrix is the method of choice because it does not require any additional matrix-vector products. However, when a dangerous eigenvalue is present, the upper Hessenberg matrix in the Arnoldi decomposition of $s(A)$ typically has norm $\|H_k\|=\mathcal{O}(d^{-1})$: this makes the accurate calculation of the remaining target eigenvalues challenging for standard dense solvers. To focus on the accuracy in the computed basis $Q_k$, we work with $A_k$, but we revisit $H_k$ at the end of this section.

In keeping with the analysis in~\cref{sec:dang_eigvals,sec:twice_is_enough}, we can understand the accuracy in the computed orthonormal basis $\hat Q_k$ through the conditioning of the matrix
\begin{equation}
Y_k = \begin{bmatrix}
q_1 & \cdots & q_{k-1} & y_k
\end{bmatrix}, \qquad k=2,3,4,\ldots.
\end{equation}
The matrix $Q_k$ from the Arnoldi iterations is precisely the QR factorization of $Y_k$ obtained by orthogonalizing $y_k$ against the previous $(k-1)$ columns, which are already an orthonormal set. If $y_k$ is not too closely aligned with ${\rm span}(q_1,\ldots,q_{k-1})$, then the matrix $Y_k$ is well-conditioned, at least after a simple column scaling. Consequently, $Q_k={\rm qf}(Y_k)$ is not too sensitive to perturbations caused by round-off in $Y_k$, as discussed in~\cref{sec:sensitive_ONB}. However, if $y_k$ is closely aligned with any of the previous columns, the smallest singular value of $Y_k$ will be close to zero and $Q_k$ will be very sensitive to round-off in $Y_k$.

This perspective provides an explanation for the stagnation observed in~\cref{fig:arnoldi_fsi}. When $q_1$ is chosen randomly, $v_1^*q_1$ is generically $\mathcal{O}(1)$ (as $d\rightarrow 0$). After applying the shift-and-invert filter, we calculate (as usual) that
$$
y_2=s(A)q_1=\frac{v_1^*q_1}{de^{i\theta}}v_1 + \mathcal{O}(1).
$$
After we orthogonalize $y_2$ against $q_1$ to compute $q_2$, then for some constant $h_2$, we have
\begin{equation}\label{eqn:keyStag_Arnoldi}
q_2=h_2\frac{v_1^*q_1}{de^{i\theta}}(v_1-(v_1^*q_1)q_1) + \mathcal{O}(1).
\end{equation}
In other words, $q_2$ may not be dominated by $v_1$, but ${\rm span}(q_1,q_2)$ contains approximations to $v_1$ that are accurate to $\mathcal{O}(d)$. 

Now, note that $q_2$ is not near orthogonal to $v_1$ unless $q_1$ happens to be very closely aligned with $v_1$. This means that the subsequent iterate $y_3$ is also aligned with $v_1$, and therefore with a vector in ${\rm span}(q_1,q_2)$, to about order $d$. Consequently, the matrix $Y_3$ is ill conditioned and we expect that $Q_3$, and in particular $q_3$, can only be accurate to about order $u/d$ when computed in floating-point precision. Moreover, $q_3$ is not dominated by $v_1$ and this process repeats, so that each iterate $y_k$ is closely aligned with $v_1$ in ${\rm span}(q_1,q_2)$, leading to errors in $q_k$ on the order of $u/d$.

\begin{figure}[!tbp]
  \centering
  \begin{minipage}[b]{0.48\textwidth}
    \begin{overpic}[width=\textwidth]{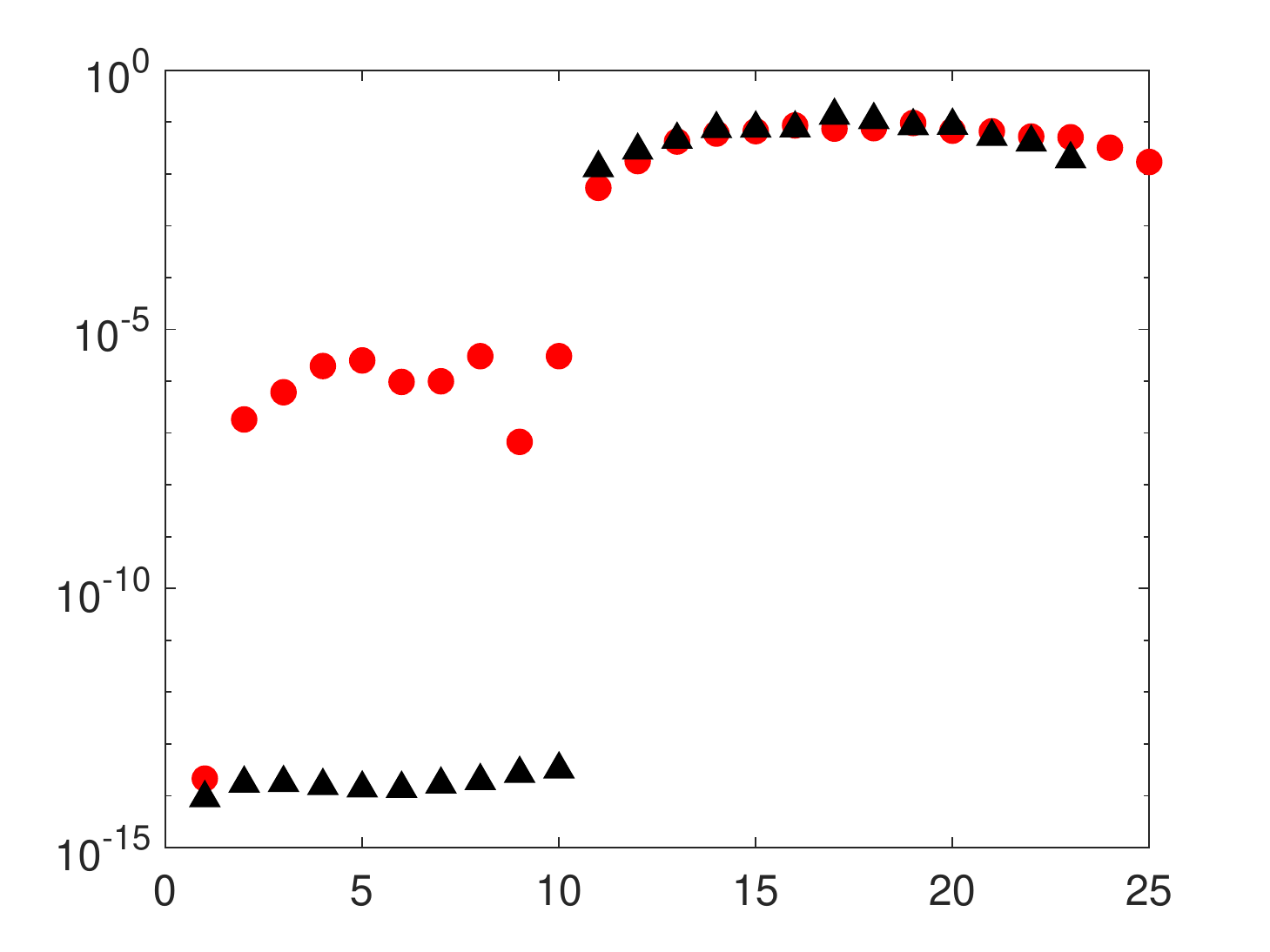}
   	 \put (34,73) {$\displaystyle \|A\hat v_i-\hat \lambda_i\hat v_i\|$}
     \put (50,-2) {$\displaystyle i$}
     \end{overpic}
  \end{minipage}
  \hfill
  \begin{minipage}[b]{0.48\textwidth}
    \begin{overpic}[width=\textwidth]{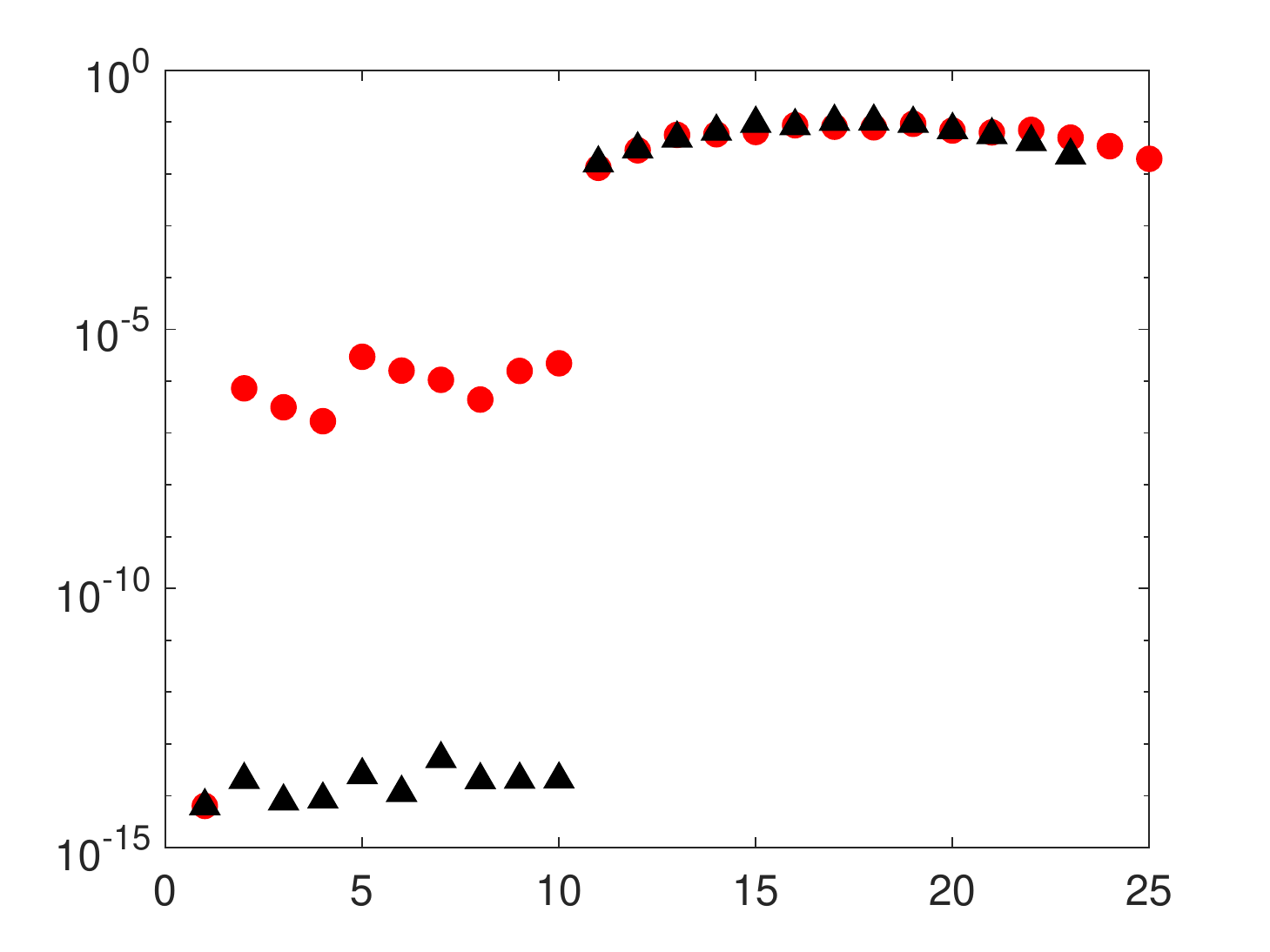}
   	 \put (34,73) {$\displaystyle \|A\hat v_i-\hat \lambda_i\hat v_i\|$}
     \put (50,-2) {$\displaystyle i$}
     \end{overpic}
  \end{minipage}
  \caption{After restarting Arnoldi with Ritz vectors that are nearly orthogonal to the dangerous direction, Arnoldi produces approximations to the target eigenpairs with accuracy near the unit round-off. Both plots compare eigenpair residuals after $25$ steps of shift-and-invert Arnoldi with no restart (red circles) to eigenpair residuals obtained after $25$ total steps of shift-and-invert Arnoldi with the Ritz restart. The eigenpairs were extracted from $Q^*_{25}AQ_{25}$ in the left panel and from the Hessenberg matrix $H_{25}$ in the right panel.
\label{fig:exp4_results}}
\end{figure}

In our discussion above, note that $y_3$ was only aligned with $v_1$, and thus close to ${\rm span}(q_1,q_2)$, because $q_2$ was not nearly orthogonal to $v_1$. Unlike in subspace iteration, the dangerous direction is never rendered harmless by orthogonalizing directly against it! The geometric picture of the iterates $y_2,y_3,y_4,\ldots$ being attracted to $v_1$ as a result of $q_2,q_3,q_4,\ldots$ not being sufficiently orthogonal to $v_1$ suggests an interesting fix. If we restart the Arnoldi iteration with the Ritz approximation associated to $v_1$ after the second iteration, the picture changes drastically. Again, $y_2$ is aligned with $v_1$, but now it is orthogonalized against $q_1=v_1+\mathcal{O}(d)$. The corresponding $q_2$ may not be particularly accurate, but this doesn't matter much: the point is that all subsequent iterates are orthogonalized against the dangerous direction (via $q_1$) up to order $\mathcal{O}(d)$. Analogous to the situation encountered in subspace iteration, the iterates $y_3,y_4,y_5,\ldots,$ are no longer dominated by $v_1$ and, consequently, $q_3,q_4,q_5,\ldots$ can be computed accurately. In a sense, we are tricking Arnoldi into running in the orthogonal complement of the dangerous direction.

\Cref{fig:exp4_results} demonstrates this restart strategy in action. As we saw earlier, $25$ iterations of shift-and-invert Arnoldi leads to stagnation in $9$ of the $10$ target eigenpairs. However, we can resolve all $10$ target eigenpairs to unit round-off accuracy in $25$ iterations if we restart with the Ritz vector corresponding to the dangerous direction after the second iteration. The right Ritz vector is easy to identify: it is most closely aligned with the second iteration $y_2$. It is worth noting that the Ritz restart strategy seems to be equally successful when eigenpairs are extracted from the Hessenberg matrix $H_k$ instead of $Q_k^*AQ_k$ (see the right panel in~\cref{fig:exp4_results}).

\section{Multiple dangerous eigenvalues}\label{sec:num_exp}

For simplicity, our analysis has focused on the case where there is just one dangerous eigenvalue. However, other situations may arise more naturally in practice. When eigenvalues are heavily clustered, many dangerous eigenvalues may surround a single pole at various distances. The main message of our results carries over to these cases. To illustrate this, we generate a $200\times 200$ symmetric matrix with 15 target eigenvalues in $[10,15]$, and employ a filter with equally-spaced poles on a circular contour centered at $12.5$. There are two dangerous eigenvalues at $10+10^{-13}$, and the other 13 target eigenvalues are clustered exponentially at the pole, taking the values $10+10^{-i},i=0,1,2,\ldots,12$. Thus, there are two dangerous eigenvalues, along with many less harmful but still dangerous eigenvalues.~\Cref{fig:mult} shows the results with two rational filters: one excellent and one of medium quality. Just as in \cref{sec:dang_eigvals,sec:twice_is_enough}, we see that twice is enough if the filter quality is high; with a poorer filter, the iterates beyond the second behaves as if there was no dangerous eigenvalue (also, see the right panels in~\cref{fig:arnoldi_fsi,fig:contour_refine}). 
\begin{figure}[!tbp]
  \centering
  \begin{minipage}[b]{0.48\textwidth}
    \begin{overpic}[width=\textwidth]{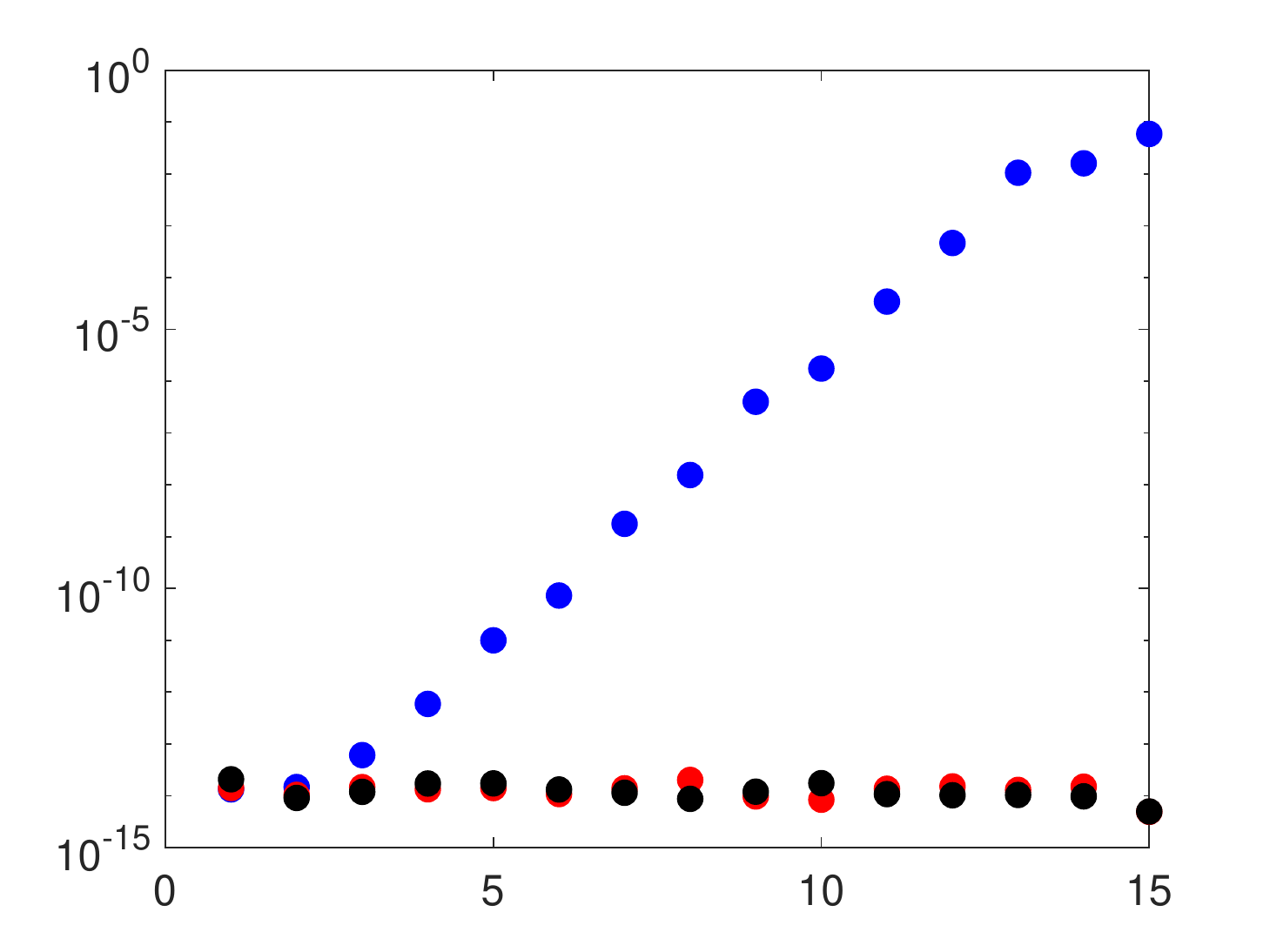}
     \put (34,73) {$\displaystyle \|A\hat v_i-\hat \lambda_i\hat v_i\|$}
     \put (92,8) {$\displaystyle k=3$}
   	 \put (92,13) {$\displaystyle k=2$}
   	 \put (92,63) {$\displaystyle k=1$}
     \put (50,-2) {$\displaystyle i$}
     \end{overpic}
  \end{minipage}
  \hfill
  \begin{minipage}[b]{0.48\textwidth}
    \begin{overpic}[width=\textwidth]{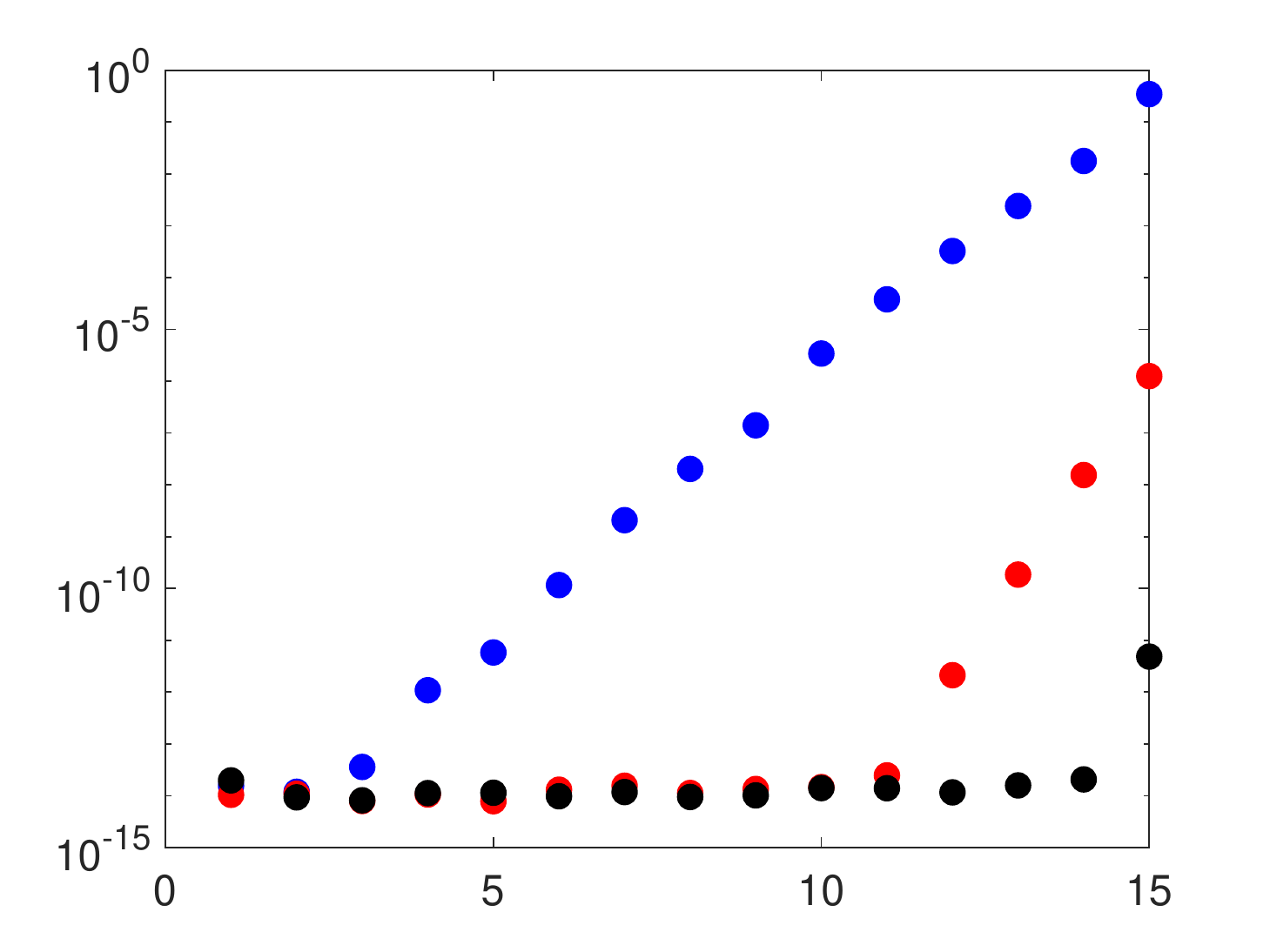}
     \put (34,73) {$\displaystyle \|A\hat v_i-\hat \lambda_i\hat v_i\|$}
     \put (92,20) {$\displaystyle k=3$}
   	 \put (92,42) {$\displaystyle k=2$}
   	 \put (92,65) {$\displaystyle k=1$}
     \put (50,-2) {$\displaystyle i$}
     \end{overpic}
  \end{minipage}
  \caption{Convergence with multiple dangerous eigenvalues. On the left, two iterations of rational subspace iteration with a high-quality filter ($\ell=32$ poles) reduce the residuals of $15$ target eigenpairs to the order of $u$, despite exponential clustering of the target eigenvalues at a pole. On the right, three iterations of rational subspace iteration with a medium-quality filter ($\ell=8$) reduce the residuals of $15$ target eigenpairs geometrically (see~\cref{thm:perturbed_convergence}), with no observable interference from the exponentially clustered eigenvalues.
  \label{fig:mult}}
\end{figure}
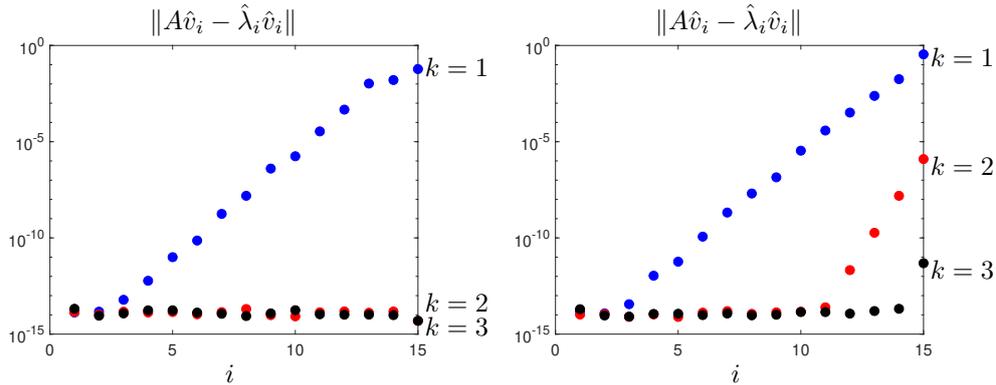

\section*{Conclusions}
Subspace and Arnoldi iterations can be extremely efficient and flexible tools for computing a few target eigenpairs when accelerated with a rational filter, but one must be cautious about eigenvalues near the poles. The damage incurred by such dangerous eigenvalues is confined to the first iteration of subspace iteration: subsequent iterations self-correct and the eigenpairs are computed to machine precision as orthogonalization effectively deflates the dangerous direction. If the matrix is real-symmetric or, more generally, normal, then the influence of the dangerous eigenvalue is corrected in just two iterations. For matrices whose eigenvectors are not orthogonal (or very near to orthogonal), self-correction occurs geometrically over a series of iterations at a rate of roughly $u/d$ in the best case (it is possible that non-normal effects cause instability in the worst case). For Arnoldi and similar Krylov schemes, we recommend restarting the iteration with the Ritz approximation to the dangerous eigenvector in order to resolve all target eigenpairs to full precision.

\section*{Acknowledgements}
We would like to thank Alex Townsend for encouraging us to investigate the stability of contour integral eigensolvers when an eigenvalue is near a quadrature node, as well as for his careful reading of an early draft.

\bibliography{draft}
\bibliographystyle{plain}
\end{document}